\documentclass[a4paper,11pt]{article}

\usepackage[top=3.2cm, bottom=3cm, left=3.45cm, right=3.45cm]{geometry}
\usepackage[T1]{fontenc}
\usepackage[utf8]{inputenc}
\usepackage{latexsym}
\usepackage{amsmath,amsthm,amssymb,amsfonts}
\usepackage{mathtools}
\usepackage[english]{babel}
\usepackage{tikz}
\usetikzlibrary{patterns, matrix, shapes}
\tikzstyle{NE-lines}=[pattern=north east lines, pattern color=black!45]
\usepackage{tikz-cd}
\usepackage{booktabs}
\usepackage{soul} 
\usepackage{diagbox}

\DeclareMathOperator{\img}{Im}

\newcommand{\quand}{\quad\text{and}\quad}

\mathchardef\mhyphen="2D
\def\loongmapsto#1{%
  \begin{tikzpicture}
    \draw (0,0.5mm) -- (0,-0.5mm);
    \draw[->] (0,0) -- (1.4, 0) node[above,midway] {\scriptsize #1};
  \end{tikzpicture}
}
\newcommand{\Cay}{\mathrm{Cay}}           
\newcommand{\Sym}{\mathrm{Sym}}           
\newcommand{\RGF}{\mathrm{RGF}}           
\newcommand{\Ascseq}{\mathcal{A}}         
\newcommand{\Modasc}{\hat{\Ascseq}}       
\newcommand{\F}{F}                        
\newcommand{\fishpattern}{\mathfrak{f}}   
\newcommand{\Bur}{\mathrm{Bur}}           
\newcommand{\Des}{\mathrm{Des}}           
\newcommand{\asctops}{\mathrm{asctops}}   
\newcommand{\nub}{\mathrm{nub}}           
\newcommand{\sort}{\mathrm{sort}}         
\newcommand{\identity}{\mathrm{id}}       
\newcommand{\WI}{I}                       
\newcommand{\twoplustwo}{\mathbf{2\hspace{-0.2em}+\hspace{-0.2em}2}}   

\newtheorem{theorem}{Theorem}[section]
\newtheorem{theorem*}{Theorem}[section]
\newtheorem{proposition}[theorem]{Proposition}
\newtheorem{lemma}[theorem]{Lemma}
\newtheorem{corollary}[theorem]{Corollary}

\newtheorem*{openproblem*}{Open Problem}
\newtheorem{conjecture}[theorem]{Conjecture}
\theoremstyle{definition}

\newtheorem{remark}[theorem]{Remark}
\newtheorem*{remark*}{Remark}
\newtheorem{example}{Example}
\newtheorem*{example*}{Example}

\title{Modified ascent sequences and Bell numbers}
\author{Giulio Cerbai
\thanks{Science Institute, University of Iceland, Iceland}
}
\date{}

\begin{document}

\maketitle

\begin{abstract}
In 2011, Duncan and Steingr\'imsson conjectured that modified ascent
sequences avoiding any of the patterns~$212$, $1212$, $2132$, $2213$,
$2231$ and $2321$ are counted by the Bell numbers. Furthermore,
the distribution of the number of ascents is the reverse of the
distribution of blocks on set partitions.
We solve the conjecture for all the patterns except~$2321$.
We describe the corresponding sets of Fishburn permutations
by pattern avoidance, and leave some open questions for future work.
\end{abstract}

\section{Introduction}
\thispagestyle{empty}

In recent years, several papers have been devoted to the study of
combinatorial structures enumerated by the Fishburn numbers.
The sequence of Fishburn numbers is recorded as A022493 in the
OEIS~\cite{Sl} and has the elegant generating function~\cite{Za}
$$
\sum_{n\geq 0}\prod_{k=1}^n\bigl(1-(1-t)^k\bigr)
= 1 + t + 2t^2 + 5t^3 + 15t^4 + 53t^5 + 217t^6 + \cdots
$$
The milestone paper~\cite{BMCDK} by Bousquet-M\'elou, Claesson, Dukes
and Kitaev introduced bijections between four families of Fishburn-type
objects. Among them, we find ascent sequences, Fishburn permutations and unlabeled $(\twoplustwo)$-free posets.
The set of Fishburn permutations consists of permutations avoiding a certain bivincular pattern.
They can be constructed inductively by successive insertion of a new
maximum value; by recording the positions where the new maximum is
inserted at each step, we obtain ascent sequences.
Later, Dukes and Parviainen~\cite{DP} showed a recursive bijection
between ascent sequences and upper triangular matrices with nonnegative
integer entries and no null rows or columns, the Fishburn matrices.
Modified ascent sequences were introduced to better clarify the relation
between ascent sequences and $(\twoplustwo)$-free posets.
Originally~\cite{BMCDK}, they were defined as the bijective image of
ascent sequences under a certain hat mapping.
More recently, Claesson and the current author~\cite{CC} characterized
them independently as Cayley permutations where ascent tops
and leftmost copies coincide.
The same authors~\cite{CC2} introduced Fishburn trees to simplify
the bijections relating all the Fishburn-type structures mentioned
so far. They act as central objects from which modified
ascent sequences, Fishburn matrices and $(\twoplustwo)$-free posets are
obtained transparently.
Claesson and the current author~\cite{CC} also initiated the development
of a theory of transport of patterns between Fishburn permutations and modified ascent sequences.
Their framework is based on the fact that the map relating these
two structures can be described by Burge transpose of Burge words,
an operation that behaves well with respect to pattern containment.
Pattern avoidance on (modified) ascent sequences and Fishburn permutations
have been discussed recently by other authors~\cite{DZ,DS,Eg,GW}.

The current paper, which fits in the same line of research, is mostly
devoted to the proof of the following conjecture proposed by Duncan
and Steingr\'imsson~\cite{DS} in 2011:

\begin{conjecture}\label{conj_ds}
On modified ascent sequences, the patterns
$$
212,\;1212,\;2132,\;2213,\;2231,\;2321
$$
are all Wilf-equivalent and the enumeration of
modified ascent sequences avoiding any of these patterns is given by the
Bell numbers. Moreover, the distribution of the number of ascents is
the reverse of the distribution of blocks on set partitions.
\end{conjecture}

The Bell numbers appear as sequence A000110 in the OEIS~\cite{Sl}:
$$
1, 1, 2, 5, 15, 52, 203, 877, 4140, 21147, 115975.
$$
Note that $12132$ is the only modified ascent sequence of length~$5$
that contains any of the patterns listed in Conjecture~\ref{conj_ds};
hence there are $52$ such sequences instead of $53$.
It is well known that the $n$th Bell number is equal to the number
of set partitions over $\lbrace 1,2,\dots,n\rbrace$, and the distribution
of blocks is given by the Stirling numbers of the second kind.

In Section~\ref{sec_prel}, we provide the necessary tools and definitions.
We define restricted growth functions and show that restricted
growth functions of length $n$ with maximum value $k$ encode
bijectively set partitions of $\lbrace 1,2,\dots,n\rbrace$ with
$k$ blocks.
We define modified ascent sequences and Fishburn permutations, as well
as the operation of Burge transpose that relates them.
We define classical patterns and mesh patterns, both on classical
permutations and on permutations with repeated entries, i.e.\ Cayley
permutations. We end the section with some general results concerning
the enumerative aspects of modified ascent sequences.

In Section~\ref{section_212}, we solve Conjecture~\ref{conj_ds} for
the first three patterns $212$, $1212$ and $2132$.
To start, we show that in fact the avoidance of any pattern in
$\lbrace 212,1212,2132,12132\rbrace$ determines the same
set of modified ascent sequences, adding $12132$ to the conjecture
in the process.
Then we find a recursive construction of $212$-avoiding
modified ascent sequences by successive insertion of new maxima.
This construction embodies well known equations defining
the Stirling numbers of the second kind, and a proof of
Conjecture~\ref{conj_ds} follows immediately.
We also show that restricted growth functions encode active sites
of $212$-avoiding modified ascent sequences in a way that is similar
to how ascent sequences encode active sites of Fishburn permutations.
En passant, we obtain a bijection between these two structures.
An alternative bijection is then obtained by slightly
tweaking the Burge transpose. Quite surprisingly, the same construction
seems to be working for the pattern $2213$ as well, a fact that we were
not able to prove.

In Section~\ref{section_2213_2231}, we solve the patterns $2213$ and $2231$.
We provide a recursive construction of $2213$- and $2231$-avoiding
modified ascent sequences in terms of the number of copies of $1$
they contain. The resulting equations are different from the one
obtained previously for the pattern~$212$, but they still lead to
the Bell numbers. A bijection with the set of restricted growth
functions is once again an immediate outcome of our approach.

In Section~\ref{section_fish_perms}, we describe by pattern avoidance the
sets of Fishburn permutations corresponding to modified ascent
sequences avoiding $212$, $2213$ and $2321$, respectively.
A description of the set corresponding to $2231$-avoiding
modified ascent sequences remains to be determined.
More suggestions for future work can be found in
Section~\ref{sec_final_remarks}.

\section{Preliminaries}\label{sec_prel}

Let $n$ be a natural number. An \emph{endofunction} of size $n$
is a map $x:[n]\to[n]$, where $[n]=\lbrace 1,2,\dots, n\rbrace$.
We often identify the endofunction $x$ with the word $x=x_1\dots x_n$,
where $x_i=x(i)$ for each $i\in[n]$.
If $n=0$, we allow the empty endofunction~$\epsilon$.
A \emph{Cayley permutation}~\cite{Ca,MF} is an endofunction
$x:[n]\to[n]$ such that $\img(x)=[k]$, for some $k\le n$.
Alternatively, a nonempty endofunction $x=x_1\cdots x_n$ is a Cayley permutation if it contains
at least a copy of every integer between $1$ and its maximum
value.
In this paper, given a set $A$ whose elements are equipped with a notion
of size, we will denote by $A_n$ the set of elements in $A$ that
have size $n$. Conversely, given a definition of $A_n$ (of elements
of size $n$) we let $A=\cup_{n\geq 0}A_n$.
For instance, we let $\Cay_n$ be the set of Cayley permutations
of size $n$ and $\Cay=\cup_{n\geq 0}\Cay_n$ be the set of all
Cayley permutations.
For $n\le 3$, we have $\Cay_0=\{\epsilon\}$,
$\Cay_1=\left\lbrace 1\right\rbrace$,
$\Cay_2=\left\lbrace 11,12,21\right\rbrace$ and
$$
\Cay_3=
\left\lbrace
111,112,121,122,123,132,211,212,213,221,231,312,321
\right\rbrace.
$$
It is well known that a Cayley permutation $x$ encodes the ballot
(ordered set partition) $B_1\dots B_k$, where $i\in B_{x_i}$
and $k=\max(x)$.
For instance,
$$
x=311241334
\quad\text{encodes the ballot}\quad
\lbrace 2,3,6\rbrace
\lbrace 4\rbrace
\lbrace 1,7,8\rbrace
\lbrace 5,9\rbrace.
$$
This encoding of ballots by Cayley permutations is bijective,
and thus Cayley permutations are counted by the Fubini numbers
(sequence A000670 in the OEIS~\cite{Sl}).
If we apply the same encoding to set partitions, listed
with minima of blocks in increasing order, we obtain
\emph{restricted growth functions}
$$
\RGF_n=
\left\lbrace
x_1\cdots x_n:\;
x_1=1,\;x_{i+1}\le\max(x_1\cdots x_i)+1\;
\text{ for each $i<n$}\right\rbrace.
$$
As a result, $|\RGF_n|$ is equal to the $n$th Bell number;
furthermore, the number of restricted growth
functions of size $n$ whose maximum value is $k$ equals the number of
set partitions of $[n]$ with $k$ blocks. These are counted by the
$(n,k)$th Stirling number of the second kind, denoted here by $S(n,k)$.
For $n=3$, we obtain
$$
\RGF_3=\{111,112,121,122,123\}.
$$
The shortest restricted growth function that is not a modified
ascent sequence (defined in Section~\ref{section_intro_fish})
is $1212$; on the other hand, $1312$ is a modified ascent
sequence, but not a restricted growth function.

Let $x\in\Cay_n$ and $y\in\Cay_k$ be two Cayley permutations,
with $k\le n$. Then $x$ \emph{contains} $y$ if $x$ contains a subsequence
$x_{i_1}x_{i_2}\cdots x_{i_k}$, with $i_1<i_2<\cdots<i_k$,
which is \emph{order isomorphic} to $y$; that is, $x_{i_s}<x_{i_t}$
if and only if $y_s<y_t$ and $x_{i_s}=x_{i_t}$ if and only if $y_s=y_t$.
In this case, we write $y\le x$ and $x_{i_1}x_{i_2}\cdots x_{i_k}\simeq y$;
the subsequence $x_{i_1}x_{i_2}\cdots x_{i_k}$ is called an
\emph{occurrence} of the \emph{pattern} $y$ in $x$.
Otherwise, we say that $x$ \emph{avoids} $y$.
We denote by $\Cay(y)$ the set of Cayley permutations that avoid $y$.
A notable example is the set of \emph{permutations} $\Sym=\Cay(11)$;
equivalently, $\Sym$ is the set of bijective endofunctions.
More generally, when $B$ is a set of patterns, $\Cay(B)$ shall denote
the set of Cayley permutations avoiding every pattern in $B$.
We use analogous notations for subsets of $\Cay$.
For instance, $\Modasc(212)$ denotes the set of modified ascent
sequences (defined in Section~\ref{section_intro_fish}) avoiding
the pattern $212$.
For a detailed introduction to permutation patterns we
refer to Bevan's note~\cite{Be}.

A more general notion of containment is obtained via
\emph{mesh patterns}~\cite{Cl} and \emph{Cayley-mesh patterns}~\cite{Ce}.
A mesh pattern is a pair $(y,R)$, where $y\in\Sym_k$ is a permutation
(classical pattern) and $R\subseteq\left[0,k\right]\times\left[0,k\right]$
is a set of pairs of integers. The pairs in $R$ identify the lower left
corners of unit squares in the plot of $x$ which specify forbidden regions.
An occurrence of the mesh pattern $(y,R)$ in the permutation $x$ is an
occurrence of the classical pattern $y$ such that no other points of $x$
occur in the forbidden regions specified by $R$.
Cayley-mesh patterns, i.e. mesh patterns on Cayley permutations,
are defined analogously, but with additional regions that allow the
possibility of having repeated entries.
In this paper, we will often define mesh patterns (both on permutations
and Cayley permutations) by plotting the underlying classical pattern,
with the forbidden regions shaded.
For instance, the mesh patterns that characterize Fishburn permutations
and modified ascent sequences (see Section~\ref{section_intro_fish}) are
illustrated in Figure~\ref{mesh_patterns_modasc}; here, $\fishpattern$
would be more extensively defined as
$$
\fishpattern=\left(231,R\right),
\quad\text{with }
R=\left\lbrace (1,0),(1,1),(1,2),(1,3),(0,1),(2,1),(3,1)\right\rbrace.
$$
Note that the shaded regions of $\fishpattern$ consist of rows or columns
only; that is, if a box is shaded, then its whole row or column is shaded
as well. A mesh pattern that satisfies this property is called a
\emph{bivincular pattern}~\cite{BMCDK}. Roughly speaking, shaded
columns impose a constraint of adjacency on the positions, while
shaded rows impose a constraint of adjacency on the values.
An occurrence of $\fishpattern$ is an occurrence of $231$ where
the (entries playing the role of) $2$ and $3$ are in consecutive
positions, and the $2$ and $1$ are consecutive in value.
As an example, the permutation $31524$ avoids~$\fishpattern$ but
contains~$231$, while the permutation $32541$ contains an occurrence
of~$\fishpattern$ realized by the entries $251$.

\begin{figure}
$$
\fishpattern \;=\;
\begin{tikzpicture}[scale=0.40, baseline=20.5pt]
\fill[NE-lines] (1,0) rectangle (2,4);
\fill[NE-lines] (0,1) rectangle (4,2);
\draw [semithick] (0.001,0.001) grid (3.999,3.999);
\filldraw (1,2) circle (6pt);
\filldraw (2,3) circle (6pt);
\filldraw (3,1) circle (6pt);
\end{tikzpicture}
\qquad
\mathfrak{a} \,=\,
\begin{tikzpicture}[scale=0.50, baseline=19pt]
\fill[NE-lines] (2,0) rectangle (3,3);
\draw [semithick] (0,0.85) -- (4,0.85);
\draw [semithick] (0,1.15) -- (4,1.15);
\draw [semithick] (0,1.85) -- (4,1.85);
\draw [semithick] (0,2.15) -- (4,2.15);
\draw [semithick] (1,0) -- (1,3);
\draw [semithick] (2,0) -- (2,3);
\draw [semithick] (3,0) -- (3,3);
\filldraw (1,2) circle (5pt);
\filldraw (2,1) circle (5pt);
\filldraw (3,2) circle (5pt);
\end{tikzpicture}
\qquad
\mathfrak{b} \,=\,
\begin{tikzpicture}[scale=0.50, baseline=19pt]
\fill[NE-lines] (1,0) rectangle (2,3);
\fill[NE-lines] (0,0.85) rectangle (0.85,1.15);
\draw [semithick] (0,0.85) -- (3,0.85);
\draw [semithick] (0,1.15) -- (3,1.15);
\draw [semithick] (0,1.85) -- (3,1.85);
\draw [semithick] (0,2.15) -- (3,2.15);
\draw [semithick] (1,0) -- (1,3);
\draw [semithick] (2,0) -- (2,3);
\filldraw (1,2) circle (5pt);
\filldraw (2,1) circle (5pt);
\end{tikzpicture}
$$
\caption{Mesh patterns such that $\F=\Sym(\fishpattern)$ and
$\Modasc=\Cay(\mathfrak{a},\mathfrak{b})$.}\label{mesh_patterns_modasc}
\end{figure}

\subsection{Modified ascent sequences and Fishburn permutations}\label{section_intro_fish}

Given a Cayley permutation $x$, let
$$
\asctops(x)= \{(1,x_1)\}\cup \{(i, x_i): 1 < i \le n,\, x_{i-1} < x_i\}
$$
be the set of \emph{ascent tops} and their indices, including the
first element; let also
$$
\nub(x) = \{(\min x^{-1}(j), j): 1\leq j\leq \max(x) \}
$$
be the set of \emph{leftmost copies} and their
indices\footnote{The name ``$\nub$'' comes from an Haskell function
that removes duplicate elements from a list, keeping only the
first occurrence of each element. One may also think of $\nub$
as a short for ``not used before''.}.
The set $\nub$ has recently played a central role in the
context of modified (difference) ascent
sequences~\cite{CeModasc,CCSmodasc,CCSselfmod}.
As an example, we have
$$
\asctops(1212)=\{(1,1),(2,2),(4,2)\}
\quad\text{while}\quad
\nub(1212)=\{(1,1),(2,2)\}.
$$
On the other hand,
$$
\asctops(1312)=\nub(1312)=\{(1,1),(2,3),(4,2)\}.
$$
When there is no ambiguity, we will sometimes abuse notation and simply
write $x_i\in\nub(x)$ or $x_i\in\asctops(x)$. 
If $x_i\in\nub(x)$ and $x_i=a$, we say that $x_i$ is the
leftmost copy of $a$ in $x$; or, that $x_i$ is a leftmost copy
in $x$.
Claesson and the current author~\cite{CC} proved the following
characterization of the set $\Modasc$ of \emph{modified ascent
sequences}\footnote{We depart slightly from the original
paper~\cite{BMCDK} where modified ascent sequences are zero-based.}, 
and we shall use it as the definition of $\Modasc$.

\begin{proposition}\cite[Lemma 2.1]{CC}\label{modasc_char}
We have
$$
\Modasc= \{x\in\Cay: \asctops(x)=\nub(x) \}.
$$
In particular, all the ascent tops have distinct values and
$\max(x)=|\asctops(x)|$.
\end{proposition}

The equation $\asctops(x)=\nub(x)$ can be equivalently expressed
with Cayley-mesh patterns as
$$
\Modasc = \Cay(\mathfrak{a},\mathfrak{b}),
$$
where $\mathfrak{a}$ and $\mathfrak{b}$ are depicted in
Figure~\ref{mesh_patterns_modasc}. Indeed, for any Cayley permutation $x$,
it is easy to see that
$$
x\text{ avoids }\mathfrak{a}
\quad\text{if and only if}\quad
\asctops(x)\subseteq\nub(x)
$$
and
$$
x\text{ avoids }\mathfrak{b}
\quad\text{if and only if}\quad
\asctops(x)\supseteq\nub(x).
$$
Indeed, the rightmost entry in any occurrence of $\mathfrak{a}$ is
an ascent top which is not a leftmost copy; and, conversely, the
rightmost entry in any occurrence of $\mathfrak{b}$ is a leftmost copy
that is not an ascent top.
The set $\Modasc$ can be alternatively defined in a recursive fashion
as follows~\cite{CC}.
There is exactly one modified ascent sequence of length zero,
the empty word, and of length one, the single letter word $1$.
For $n\geq 1$, every $y\in\Modasc_{n+1}$ is of one of two forms depending
on whether the last letter forms an ascent with the penultimate letter:
\begin{itemize}
\item $y=x_1\cdots x_n\;x_{n+1}$,\,with\, $1\le x_{n+1}\le x_n$, or
\item $y=\tilde{x}_1\cdots\tilde{x}_n\;x_{n+1}$,\,with\,
$x_n<x_{n+1}\le 1+\max(x_1\cdots x_n)$,
\end{itemize}
where $x_1\cdots x_n\in\Modasc_{n}$ and, for $i\in[n]$,
$$
\tilde{x}_i=
\begin{cases}
x_i & \text{if }x_i<x_{n+1}\\
x_i+1 & \text{if }x_i\ge x_{n+1}.
\end{cases}
$$
In other words, each modified ascent sequence $x$ gives rise to
$\max(x)+1$ modified ascent sequences of length one more,
obtained by
\begin{itemize}
\item adding a new rightmost entry less than or equal
to $\max(x)+1$;
\item and, if the newly added element is an ascent top, increasing
by one all the entries of $x$ that are greater than or equal to the
new entry.
\end{itemize}

The set of Fishburn permutations is defined as $\F=\Sym(\fishpattern)$,
where $\fishpattern$ is the bivincular pattern depicted in
Figure~\ref{mesh_patterns_modasc}.
Claesson and the current author~\cite{CC} reformulated the original
bijection~\cite{BMCDK} relating $\Modasc$ and $\F$ in terms of
Burge transpose of Burge words.
Let $\WI_n$ be the subset of $\Cay_n$ consisting of the weakly
increasing Cayley permutations
$$
\WI_n=\{u\in\Cay_n:\;u_1\le u_2\le\dots\le u_n\}.
$$
Define the set of biwords
$$
\Bur_n= \{ (u,v)\in\WI_n\times\Cay_n:\;\Des(u)\subseteq\Des(v) \},
$$
where $\Des(v) = \{ i: v_i \geq v_{i+1}\}$ is the set of weak descents
of $v$.
Biwords in $\Bur_n$ are called \emph{Burge words} due to their connection
with the Burge variant of the RSK correspondence~\cite{Bu}.
Here we will write a Burge word either as a pair $(u,v)$, or, more
extensively, as a biword (i.e. two-line array)
$$
\binom{u}{v}=
\begin{pmatrix}
u_1 & u_2 &\dots& u_n\\
v_1 & v_2 &\dots& v_n
\end{pmatrix}.
$$
The \emph{Burge transpose} of the Burge word $(u,v)$ is the biword
$(u,v)^T$ obtained by turning each column of $(u,v)$ upside down,
and sorting the columns of the resulting biword in ascending order with
respect to the top entry, breaking ties by sorting in descending order
with respect to the bottom entry.
The reason why this operation is called a transposition lies in the
fact that Burge words are in bijection with \emph{Burge matrices},
i.e. matrices with nonnegative integer entries whose every row and column
has at least one nonzero entry. Under this bijection, the operation of
transposing a Burge word corresponds  to the usual matrix
transposition; that is, if $M$ is the matrix associated with $(u,v)$,
then the transpose of $M$ is associated with $(u,v)^T$.
Now, it is possible to show~\cite{CC} that $\Bur_n$ is closed
under transpose and $w=(u,v)$ is a Burge word if and only if $(w^T)^T=w$. 
Furthermore, if $x$ is a modified ascent sequence of length $n$
and $\identity=12\cdots n$ denotes the identity permutation
(of the same length), then the Fishburn permutation $p$ corresponding
to $x$ can be computed as
$$
\binom{\identity}{x}^T=\binom{\sort(x)}{p},
$$
where $\sort(x)$ is obtained by sorting the entries of $x$ in weakly
incrasing order. Define the map $\gamma:\Cay\to\Sym$ accordingly
by letting, for $x\in\Cay$,
$$
\binom{\identity}{x}^T=\binom{\sort(x)}{\gamma(x)}.
$$
Then, the restriction of $\gamma$ to $\Modasc$ is a size-preserving
bijection from $\Modasc$ to $\F$.
As an example to illustrate this construction, let $x=141233551$.
Note that $x$ is a modified ascent sequence, e.g. by
Proposition~\ref{modasc_char}, since
$$
\asctops(x)=\left\lbrace (1,1),(2,4),(4,2),(5,3),(7,5)\right\rbrace=\nub(x).
$$
Finally, we have
$$
\binom{\identity}{x}^T
=\binom{1\;2\;3\;4\;5\;6\;7\;8}{1\;4\;1\;2\;3\;3\;5\;1}^T
=\binom{1\;1\;1\;2\;3\;3\;4\;5}{8\;3\;1\;4\;6\;5\;2\;7}
$$
and the Fishburn permutation corresponding to $x$ is $\gamma(x)=83146527$.
It is easy to check that $\gamma(x)$ is in fact a Fishburn permutation,
that is, it avoids $\fishpattern$.

The following theorem describes how pattern avoidance
is transported from Fishburn permutations to modified
ascent sequences.

\begin{theorem}\cite[Theorem 5.1]{CC}\label{transport_theorem_modasc_fish}
For every permutation $p$, we have
$$
\F(p)=\gamma\bigl(\Modasc(B_p)\bigr),
$$
where $B_p=\lbrace x\in\Cay:\;\gamma(x)=p\rbrace$ is the
\emph{Fishburn basis} of $p$.
\end{theorem}

Theorem~\ref{transport_theorem_modasc_fish} shows that the set $\F(p)$
of Fishburn permutations avoiding $p$ is associated via the bijection
$\gamma$ with the set $\Modasc(B_p)$ of modified ascent sequences
avoiding every pattern in $B_p$. A method to construct $B_p$
explicitly can be found in the same paper~\cite{CC}.

\subsection{General results on $\Modasc$}\label{section_prop_MA}

In this section we prove some properties of the set $\Modasc$ that
will be useful later.

\begin{proposition}\label{occ_of_max_are_cons}
In a modified ascent sequence, all the copies of its maximum value
are in consecutive positions.
\end{proposition}
\begin{proof}
Let $x\in\Modasc$ and let $m=\max(x)$. We show that all the entries
of $x$ that are equal to $m$ are in consecutive positions.
Equivalently, we show that, if $x_i=x_j=m$ for some $i<j$,
then it must be $x_{\ell}=m$ for each $i<\ell<j$.
For a contradiction, suppose that there is at least one entry between
$x_i$ and $x_j$ that is smaller than $m$, and let
$$
\ell=\max\lbrace k:\;i<k<j\text{ and }x_k<m\rbrace
$$
be the index of the rightmost such entry.
Then $x_{\ell+1}=m$, $x_{\ell}<x_{\ell+1}$ and $x_ix_{\ell}x_{\ell+1}$
is an occurrence of $\mathfrak{a}$ in $x$, which is impossible due
to Proposition~\ref{modasc_char}.
\end{proof}

Given a natural number $n$, let $f_n=|\Modasc_n|$ denote the
$n$th Fishburn number. For $1\le k\le n$, let
$$
\Modasc_n(k)
=\lbrace x\in\Modasc_n:\;\text{$x$ contains $k$ copies of $\max(x)$}\rbrace
$$
and let $f_n(k)=|\Modasc_n(k)|$. Note that $f_n=\sum_{k=1}^n f_n(k)$.

\begin{proposition}\label{1_occ_of_max}
For $n\ge 1$, we have
$$
f_{n+1}(1)=f_{n+1}-f_n.
$$
\end{proposition}
\begin{proof}
By Proposition~\ref{occ_of_max_are_cons}, in a modified
ascent sequence all the entries that are equal to its maximum value
are in consecutive positions. As a result, for $k\ge 2$, the removal of
the rightmost such entry yields a bijection from $\Modasc_{n+1}(k)$ to
$\Modasc_n(k-1)$. Thus $f_{n+1}(k)=f_n(k-1)$ and also
\begin{align*}
f_{n+1}&=\sum_{k=1}^{n+1}f_{n+1}(k)\\
&=f_{n+1}(1)+\sum_{k=2}^{n+1}f_{n+1}(k) & [f_{n+1}(k)=f_n(k-1)]\\
&=f_{n+1}(1)+\sum_{k=2}^{n+1}f_{n}(k-1) & [i=k-1]\\
&=f_{n+1}(1)+\sum_{i=1}^{n}f_{n}(i)\\
&=f_{n+1}(1)+f_n,
\end{align*}
from which $f_{n+1}(1)=f_{n+1}-f_n$ follows immediately.
\end{proof}

\begin{corollary}
For $n\ge 1$, we have
$$
f_n=\sum_{i=1}^n f_i(1);
$$
that is, Fishburn numbers are the partial sum of the sequence
$\lbrace f_n(1)\rbrace_{n\ge 1}$.
\end{corollary}
\begin{proof}
The case $n=1$ holds since $f_1=f_1(1)$. For $n\ge 2$, we repeatedly iterate Proposition~\ref{1_occ_of_max} to obtain
$$
f_n=f_n(1)+f_{n-1}
=f_n(1)+f_{n-1}(1)+f_{n-2}
=\dots
=\sum_{i=1}^nf_i(1).
$$
\end{proof}

The first ten terms of the sequence $\lbrace f_n(1)\rbrace_{n\ge 1}$ are
$$
1,1,3,10,38,164,797,4321,25905,170368.
$$
We call these the \emph{$2$-Fishburn numbers}.
The $2$-Fishburn numbers can be alternatively expressed in terms of

\begin{equation}\label{eq_1}
f_{n,m}=|\lbrace x\in\Modasc_n:\;\max(x)=m\rbrace|
\quad\text{as}\quad
f_{n+1}(1)=\sum_{m=1}^{n}mf_{n,m}.
\end{equation}

Indeed, we have $f_n=\sum_{m=1}^nf_{n,m}$. Furthermore,
referring to the recursive construction of $\Modasc$ described
in Section~\ref{section_intro_fish}, if $x\in\Modasc_n$ has maximum
value $\max(x)=m$, then $x$ gives rise (by insertion of a new rightmost
entry) to $m+1$ modified ascent sequences of length $n+1$. Therefore,
$$
f_{n+1}=\sum_{m=1}^{n}(m+1)f_{n,m}
$$
and
\begin{align*}
f_{n+1}(1)
=f_{n+1}-f_n
=\sum_{m=1}^{n}(m+1)f_{n,m}-\sum_{m=1}^{n}f_{n,m}
=\sum_{m=1}^{n}mf_{n,m}.
\end{align*}

The triangle $f_{n,m}$ is sequence A137251 in the OEIS~\cite{Sl}.
Note that we were not able to find a combinatorial construction
that embodies Equation~\eqref{eq_1} directly.

Next we show that the insertion of a new strict maximum
between two consecutive entries of a modified ascent sequence
yields another modified ascent sequence if and only if the two entries
form a weak descent.

\begin{proposition}\label{insert_max}
Let $x\in\Modasc_n$ and let $m=\max(x)$. For $i\in[n]$, denote
by $x^{(i)}$ the sequence
$$
x^{(i)}=x_1\cdots x_i\;(m+1)\;x_{i+1}\cdots x_n
$$
obtained by inserting $m+1$ immediately after $x_i$. Then
$$
x^{(i)}\in\Modasc_{n+1}
\quad\text{if and only if}\quad
x_i\ge x_{i+1}\text{ or }i=n.
$$
\end{proposition}
\begin{proof}
It is a direct consequence of the equality $\nub(x)=\asctops(x)$
that defines $\Modasc$. Indeed, if $x_i<x_{i+1}$, then
$x_{i+1}\in\asctops(x)=\nub(x)$.
Now, due to the insertion of $m+1$, the entry $x_{i+1}$ is no longer
an ascent top in $x^{(i)}$; however, it is still a leftmost copy.
Since the other pairs of consecutive entries are not affected by the
insertion of $m+1$, the set $\nub(x^{(i)})$ is strictly contained in
$\asctops(x^{(i)})$, and thus $x^{(i)}$ is not a modified ascent sequence.
On the other hand, if $x_i\ge x_{i+1}$ or $i=n$, all the ascent tops
and all the leftmost copies of $x$ are preserved in $x^{(i)}$. Also
the new entry $m+1$ is both an ascent top and a leftmost copy.
Hence $\asctops(x^{(i)})=\nub(x^{(i)})$ and $x^{(i)}$ is a
modified ascent sequence.
\end{proof}

\section{Patterns $212,1212,2132,12132$}\label{section_212}

In this section we solve Conjecture~\ref{conj_ds} for every pattern
$y\in\lbrace 212,1212,2132\rbrace$. In fact we show that
$$
\Modasc(212)=\Modasc(1212)=\Modasc(2132)=\Modasc(12132)
$$
and add $12132$ to the list of patterns.

\begin{lemma}\label{lemma_1y_y}
Let $y=y_1y_2\cdots y_k$ be a Cayley permutation. Suppose that $y_2=1$
is the only copy of $1$ in $y$. Then
$$
\Modasc(y)=\Modasc(1y).
$$
\end{lemma}
\begin{proof}
Let $x\in\Modasc$. If $x$ avoids $y$, then it avoids $1y$.
Hence the inclusion $\Modasc(y)\subseteq\Modasc(1y)$ holds.
To prove the opposite inclusion, we show that if $x$ contains $y$,
then it contains $1y$ too.
Let $x_{i_1}x_{i_2}\cdots x_{i_k}$ be an occurrence of $y$ in $x$,
with $i_1<i_2<\dots<i_k$. Let $x_j$ be the leftmost smallest
entry between $x_{i_1}$ and $x_{i_3}$; more formally, let
$$
j=\min\lbrace
j:\;i_1<j<i_3,\;x_j=\min\left(x_{i_1}x_{i_1+1}\cdots x_{i_3}\right)
\rbrace.
$$
Note that $x_{i_1}x_jx_{i_3}\cdots x_{i_k}\simeq y$ since $x_j\le x_{i_2}$
and $y_2=1$ is the only copy of $1$ in $y$.
Let $j'$ be the index of the leftmost copy of $x_{j}$ in $x$.
Due to our choice of $j$, we have that $(j,x_{j})\notin\asctops(x)=\nub(x)$.
In particular, it must be $j'<i_1$. Hence $x_{j'}x_{i_1}x_{j}x_{i_3}\cdots x_{i_k}$
is an occurrence of $1y$ in $x$, as desired.
\end{proof}

\begin{corollary}\label{212_1212_and_2132_12132}
We have
$$
\Modasc(212)=\Modasc(1212)
\quand
\Modasc(2132)=\Modasc(12132).
$$
\end{corollary}
\begin{proof}
Both equalities follow immediately from Lemma~\ref{lemma_1y_y} by
letting $y=212$ and $y=2132$, respectively.
\end{proof}

\begin{proposition}\label{212_eq_2132}
We have
$$
\Modasc(212)=\Modasc(2132).
$$
\end{proposition}
\begin{proof}
The inclusion $\Modasc(212)\subseteq\Modasc(2132)$ is trivial.
To prove the opposite inclusion, suppose that $x$ contains an
occurrence $x_ix_jx_k$ of $212$. We show that $x$ contains $2132$.
Note that $(k,x_k)\notin\nub(x)$ since $x_i=x_k$ and $i<k$.
Thus, due to the equality $\nub(x)=\asctops(x)$ defining $\Modasc$,
it must be $x_{k-1}\ge x_k$. If $x_{k-1}>x_k$, then $x_ix_jx_{k-1}x_k$
is an occurrence of $2132$ in $x$, as wanted.
Finally, if $x_{k-1}=x_k$, then we can repeat the same argument on the
occurrence $x_ix_jx_{k-1}$ of $212$ until we eventually fall back in
the previous case.
\end{proof}

By Corollary~\ref{212_1212_and_2132_12132} and Proposition~\ref{212_eq_2132},
we have
$$
\Modasc(212)=\Modasc(1212)=\Modasc(2132)=\Modasc(12132).
$$
For the rest of this section, let $B=\Modasc(212)$.
For $n\ge 1$ and $1\le k,m\le n$, in analogy with the definitions
of $f_n(k)$ and $f_{n,m}$ given in Section~\ref{section_prop_MA},
define the sets

\begin{center}
\begin{tabular}{p{1.34cm}l}
$B_n(k)$     & $=\;\lbrace x\in B_n:\;
\text{$x$ contains $k$ copies of $\max(x)$}\rbrace$;\\
$B_{n,m}$    & $=\;\lbrace x\in B_n:\;\max(x)=m\rbrace$;\\
$B_{n,m}(k)$ & $=\;B_n(k)\cap B_{n,m}$.
\end{tabular}
\end{center}

Denote their cardinalities by
$$
\begin{matrix}
b_n=|B_n|,
& b_{n}(k)=|B_n(k)|,
& b_{n,m}=|B_{n,m}|,
& b_{n,m}(k)=|B_{n,m}(k)|.
\end{matrix}
$$
We wish to prove that $b_{n,m}=S(n,n-m+1)$, where $S(n,j)$ is
the $(n,j)$th Stirling number of the second kind.
Since in any modified ascent sequence the number of ascent tops
equals the maximum value, this settles Conjecture~\ref{conj_ds}
for $y\in\lbrace 212,1212,2132,12132\rbrace$.

Let us start by showing that if $x\in B_{n+1}(1)$ and $m=\max(x)$,
then the sequence obtained by removing the only copy of $m$ from $x$
is a modified ascent sequence.
Together with Proposition~\ref{insert_max}, this implies that every
such $x\in B_{n+1}(1)$ is obtained uniquely from an element of $B_n$
by inserting a new strict maximum between two consecutive entries that
form a weak descent.
Building upon this property, in combination with the fact that all the
copies of $\max(x)$ are in consecutive positions (by
Proposition~\ref{occ_of_max_are_cons}), we will obtain a recursive
construction for the set $B_n$.

\begin{proposition}\label{remove_max}
Let $x\in B_{n+1}(1)$. Let $x_i$ be the only copy of $\max(x)$ in $x$.
Denote by $\tilde{x}$ the sequence
$$
\tilde{x}=x_1\cdots x_{i-1}x_{i+1}\cdots x_{n+1}
$$
obtained from $x$ by removing $x_i$. Then $\tilde{x}\in B_n$.
\end{proposition}
\begin{proof}
It is easy to see that $\tilde{x}$ avoids $212$ since $x$ does so.
We shall prove that $\tilde{x}$ is a modified ascent sequence by
showing that $\asctops(\tilde{x})=\nub(\tilde{x})$.
Recall that $\asctops(x)=\nub(x)$ since $x\in\Modasc$.
In particular, we have $\max(x)=x_i>x_{i+1}$ and thus
$x_{i+1}\notin\nub(x)$.
Let $j$ be the index of the leftmost copy of $x_{i+1}$ in $x$,
so that $x_j=x_{i+1}$ and $(j,x_j)\in\nub(x)$. Note that it must be
$x_{i-1}\ge x_{i+1}$, or else we would have an occurrence
$x_jx_{i-1}x_{i+1}$ of $212$ in $x$, which is impossible.
Therefore, $x_{i-1}\ge x_{i+1}$, i.e. $x_{i+1}\notin\asctops(\tilde{x})$,
and also $x_{i+1}\notin\nub(\tilde{x})$. Since the other pairs
of consecutive elements are not affected by the removal of $x_i$, 
we have $\asctops(\tilde{x})=\nub(\tilde{x})$, which concludes
the proof.
\end{proof}

In general, the previous proposition does not hold if $x$ contains $212$.
For instance, $12132\in\Modasc_5(1)$, but $1212$ is not a modified ascent
sequence.

It is easy to see that, if $x$ avoids $212$ and $x^{(i)}$ is defined
as in Proposition~\ref{insert_max}, i.e. by inserting a new strict maximum
immediately after $x_i$, then $x^{(i)}$ avoids $212$ as well.
Indeed the new entry $m+1$, where $m=\max(x)$, can not be part of an
occurrence of $212$ in $x^{(i)}$. Therefore, by Proposition~\ref{insert_max}
and Proposition~\ref{remove_max}, each modified ascent sequence $x$
in $B_{n+1}(1)$ is obtained uniquely from some $\tilde{x}\in B_{n}$ by
inserting a new entry equal to $\max(x)+1$, either between two consecutive
entries $x_i\ge x_{i+1}$ or after $x_n$.
Recall also that $\max(x)=|\asctops(x)|$.
Therefore, if $x\in B_{n,m}$, then there are exactly $n-m+1$ positions
where the insertion of a new strict maximum yields an element
of $B_{n+1}(1)$. Thus
$$
b_{n+1,m+1}(1)=(n-m+1)b_{n,m}
\quand
b_{n+1}(1)=\sum_{m=1}^n(n-m+1)b_{n,m}.
$$
Furthermore, by Proposition~\ref{occ_of_max_are_cons},
all the copies of $\max(x)$ are in consecutive positions, hence
$$
b_{n+1,m}(k+1)=b_{n,m}(k).
$$
Putting everything together, we obtain
\begin{align*}
b_{n,m}&=\sum_{i=1}^n b_{n,m}(i)\\
&=b_{n,m}(1)+\sum_{i=2}^n b_{n,m}(i) & [b_{n,m}(1)=(n-m+1)b_{n-1,m-1}]\\
&=(n-m+1)b_{n-1,m-1}+\sum_{i=2}^n b_{n-1,m}(i-1) & [j=i-1]\\
&=(n-m+1)b_{n-1,m-1}+\sum_{j=1}^{n-1} b_{n-1,m}(j)\\
&=(n-m+1)b_{n-1,m-1}+b_{n-1,m}.
\end{align*}
Thus, the coefficients $b_{n,m}$ obey the recurrence
\begin{equation}\label{eq_stirl}
b_{n,m}=(n-m+1)b_{n-1,m-1}+b_{n-1,m}.
\end{equation}
The Stirling numbers of the second kind are defined by
$$
\begin{cases}
S(n,n)=1 & n\ge 0;\\
S(n,0)=S(0,n)=0 & n>0;\\
S(n,k)=kS(n-1,k)+S(n-1,k-1) & 0<k<n.
\end{cases}
$$
We prove that $b_{n,m}=S(n,n-m+1)$.
It is easy to see that the initial conditions are satisfied.
For instance, $b_{n,1}=S(n,n)=1$ since the only $x\in B_{n,1}$
is $x=11\cdots 1$.
Let $n\ge 2$ and $m\in [n]$; using induction on $n$ and
Equation~\eqref{eq_stirl}:
\begin{align*}
b_{n,m}&=(n-m+1)b_{n-1,m-1}+b_{n-1,m}\\
&=(n-m+1)S(n-1,n-m+1)+S(n-1,n-m) & [k=n-m+1]\\
&=kS(n-1,k)+S(n-1,k-1)\\
&=S(n,k)\\
&=S(n,n-m+1).
\end{align*}

\begin{theorem}\label{proof_212}
Let $y\in\lbrace 212,1212,2132,12132\rbrace$. Then $|\Modasc_n(y)|$
is equal to the $n$th Bell number. Furthermore, the number of
modified ascent sequences in $\Modasc_n(y)$ whose maximum value
is equal to $m$ is given by the $(n,n-m+1)$th Stirling
number of the second kind.
\end{theorem}

By Proposition~\ref{1_occ_of_max}, modified ascent sequences
that contain exactly one copy of their maximum value
are counted by the $2$-Fishburn numbers. In other words,
$f_{n+1}(1)=f_{n+1}-f_n$ and Fishburn numbers are the partial sums
of the $2$-Fishburn numbers.
Similarly, since $|B_n|$ is equal to the $n$th Bell number, the
sequence $\lbrace b_n(1)\rbrace_{n\ge 1}$ is given by the
\emph{$2$-Bell numbers} (A005493~\cite{Sl}).
This can be proved with a completely analogous argument
by observing that $b_{n+1}(k+1)=b_n(k)$. We leave the details
to the reader.

Next we define a bijection $\phi$ from $\RGF_n$ to $B_n$.
As suggested by Theorem~\ref{proof_212}, we shall in fact map
$\RGF_{n,n-m+1}$ to $B_{n,m}$, where
$$
\RGF_{n,m}=\lbrace x\in\RGF_n:\;\max(x)=m\rbrace.
$$
Indeed, under the usual encoding of set partitions by restricted
growth functions described in Section~\ref{sec_prel}, the set
$\RGF_{n,m}$ corresponds to set partitions of $[n]$ with $m$ blocks.
We proceed as follows. First we describe more directly the recursive
construction of $B_{n,m}$ embodied by Equation~\eqref{eq_stirl}.
Then we introduce the related notion of \emph{active site},
and show that restricted growth functions encode this construction
in a way that is similar to how ascent sequences encode active
sites of Fishburn permutations.

Let $x\in B_{n,m}$. As observed in the paragraph leading to
Proposition~\ref{remove_max}, if $x$ contains one copy of $m=\max(x)$,
then $x$ is obtained uniquely from some $\tilde{x}\in B_{n-1,m-1}$
by inserting $m$ either between two consecutive elements that form
a weak descent or after the last entry. Note that there are
$n-m+1$ such positions in $\tilde{x}$ since
$\max(\tilde{x})=|\asctops(\tilde{x})|$.
On the other hand, suppose that $x$ contains at least two copies
of $m$. By Proposition~\ref{occ_of_max_are_cons}, all the copies of $m$
are in consecutive positions, hence the removal of the rightmost copy
of $m$ from $x$ determines (uniquely) an element of $B_{n-1,m}$.
In other words, every $x\in B_{n,m}$ gives rise to $(n-m+1)+1$ sequences
in $B_{n+1}$:

\begin{itemize}
\item[$(i)$] The $n-m+1$ sequences in $B_{n+1,m+1}$ obtained by inserting
a new strict maximum between two entries $x_i\ge x_{i+1}$,
including the last spot after $x_n$; we call these positions the
\emph{active sites} of $x$. We also label the active sites with the
integers $1,2,\dots,n-m+1$, going from left to right.
\item[$(ii)$] The (only) sequence in $B_{n+1,m}$ obtained by inserting a new weak
maximum at the end of the string of consecutive maxima of $x$.
\end{itemize}

Similarly, given a restricted growth function $r\in\RGF_{n,m}$,
by inserting a new rightmost entry $r_{n+1}=i$ one obtains
\begin{itemize}
\item[$(i)$] $m$ sequences in $\RGF_{n+1,m}$, one for each $i=1,2,\dots,m$;
\item[$(ii)$] One sequence in $\RGF_{n+1,m+1}$, when $i=m+1$ is a new
strict maximum.
\end{itemize}

Now, given a restricted growth function $r=r_1\cdots r_n$,
define $x=\phi(r)$ inductively as follows.

\begin{itemize}
\item Map the only restricted growth function of length one, $r=1$,
to the only modified ascent sequence of length one, $x=1$.
Note that the sum $\max(r)+\max(x)$ is equal to one plus the length
of $r$ (or $x$).
\item Let $n\ge 1$ and let $\phi(r_1\cdots r_n)=x_1\cdots x_n$
be defined inductively, with $\max(r)+\max(x)=n+1$.
Let $\bar{r}=r_1\cdots r_nr_{n+1}$. We define $\bar{x}=\phi(\bar{r})$
as follows, according to whether $r_{n+1}\le\max(r)$ or
$r_{n+1}=\max(r)+1$.
\begin{enumerate}
\item If $r_{n+1}=i\le\max(r)$, then we let $\bar{x}$ be obtained
from $x$ by inserting a new strict maximum $\max(x)+1$ in the
$i$th active spot of $x$. Note that $\max(r)=n-\max(x)+1$, thus
the number of possible choices for $i$ equals the number of active
sites of $x$.
\item If $r_{n+1}=\max(r)+1$, then we let $\bar{x}$ be obtained
from $x$ by inserting a new weak maximum (at the end of the string
of consecutive maxima).
\end{enumerate}
\end{itemize}
Note that in each case we have $\max(\bar{r})+\max(\bar{x})=n+2$.
In general, if $r\in\RGF_{n,j}$ and $x=\phi(r)$, then $x\in B_{n,n+1-j}$.
As a result, the map $\phi$ defined this way is a bijection
between $\RGF_{n,n-m+1}$ and $B_{n,m}$\footnote{
It is in fact an isomorphism between the generating trees of $\RGF$
and $B$ induced by the generating rules $(i)$ and $(ii)$.}.
We have thus obtained an alternative proof of the fact that the
statistic $|\asctops(x)|=\max(x)$ on $B_n$ is equidistributed with
the reverse of the number of blocks, i.e. $\max(r)$, on set partitions
of $[n]$.
Below we illustrate the step-by-step computation of
$\phi(123224135)=141233551$. Here superscripts denote
labels of active sites, while positions between consecutive elements
that have no superscript are forbidden. At each step, we underline
the newly added element in the modified ascent sequence (on the right).

\begin{align*}
1&\quad\loongmapsto{}\quad\underline{1}{^1}\\
12&\quad\loongmapsto{}\quad 1{^1}\underline{1}{^2}\\
123&\quad\loongmapsto{}\quad 1{^1}1{^2}\underline{1}{^3}\\
1232&\quad\loongmapsto{}\quad 1{^1}1\underline{2}{^2}1{^3}\\
12322&\quad\loongmapsto{}\quad 1{^1}12\underline{3}{^2}1{^3}\\
123224&\quad\loongmapsto{}\quad 1{^1}123{^2}\underline{3}{^3}1{^4}\\
1232241&\quad\loongmapsto{}\quad 1\underline{4}{^1}123{^2}3{^3}1{^4}\\
12322413&\quad\loongmapsto{}\quad 14{^1}123{^2}3\underline{5}{^3}1{^4}\\
r=123224135&\quad\loongmapsto{}\quad 14{^1}123{^2}35{^3}\underline{5}{^4}1{^5}=\phi(r).\\
\end{align*}

\subsection{Bijection via transposition of biwords}

In the previous section, two recursive constructions of $\Modasc(212)$
and $\RGF$ lead to the definition of the bijection $\phi$.
Here we slightly tweak the Burge transpose to obtain a new bijection $\psi$.
Compared to $\phi$, the construction of $\psi$ is more straightforward,
and arguably more elegant. Despite that, the proof that $\psi$ is
a bijection come with a considerable amount of technical details.
In light of that, and also due to the fact that what we show in this
section is not strictly necessary for the rest of the paper, most of
them will be omitted.

Let $x=x_1\cdots x_n$ be a modified ascent sequence and let $m=\max(x)$.
The definition of the restricted growth function $\psi(x)$ goes as follows.
Initially, we use a step-by-step procedure to label each entry $x_j$
of $x$ with a positive integer $u_j$.
For $i=1,\dots,m$, at the $i$th step we label all the copies of $i$ in $x$.

\begin{itemize}
\item For $i=1$, we label the copies of $1$ with increasing integers
$1,2,3,\dots$, starting from the leftmost copy and going from
left to right.
\item Let $i\ge 2$. Let $t$ be the maximum label assigned at the
previous steps. Let $x_j=i$ be the leftmost copy of $i$ in $x$.
Since $\asctops(x)=\nub(x)$, we have $x_{j-1}<x_j$. In particular, the
entry $x_{j-1}$ has been labeled with $u_{j-1}$ at a previous step of
the procedure. Then we let $u_j=u_{j-1}$ and assign labels $t+1,t+2,\dots$
to the remaining copies of $i$ (going from left to right).
\end{itemize}
Finally, we arrange $x$ and the resulting labels $u=u_1\cdots u_n$ in
the biword
$$
\binom{u}{x}=\binom{u_1\;u_2\;\cdots\;u_n}{x_1\;x_2\;\cdots\;x_n}.
$$

To sum up this procedure in the most succint way,
at the $i$th step we give to the leftmost copy of $i$
the same label as the entry immediately to its left,
and give new (increasing) labels to the other copies of $i$.
An example that illustrates this construction is the following.
Let $x=141233551$. The five steps needed to determine $u$ are
illustrated below, where labels defined at each step are underlined.
\begin{align*}
\text{Step 1:}\qquad &
\begin{pmatrix}
\underline{1}& &\underline{2}& & & & & &\underline{3}\\
1&4&1&2&3&3&5&5&1
\end{pmatrix}\\
\text{Step 2:}\qquad &
\begin{pmatrix}
1& &2&\underline{2}& & & & &3\\
1&4&1&2&3&3&5&5&1\\
\end{pmatrix}\\
\text{Step 3:}\qquad &
\begin{pmatrix}
1& &2&2&\underline{2}&\underline{4}& & &3\\
1&4&1&2&3&3&5&5&1\\
\end{pmatrix}\\
\text{Step 4:}\qquad &
\begin{pmatrix}
1&\underline{1}&2&2&2&4& & &3\\
1&4&1&2&3&3&5&5&1\\
\end{pmatrix}\\
\text{Step 5:}\qquad &
\begin{pmatrix}
1&1&2&2&2&4&\underline{4}&\underline{5}&3\\
1&4&1&2&3&3&5&5&1\\
\end{pmatrix}
\end{align*}
In the end, we obtain
$$
\setlength\arraycolsep{1.5pt}
\binom{u}{x}=
\begin{pmatrix}
1&1&2&2&2&4&4&5&3\\
1&4&1&2&3&3&5&5&1\\
\end{pmatrix}.
$$
Now, define the biword $(u,x)^{T'}$ by
\begin{itemize}
\item Flipping $\binom{u}{x}\mapsto\binom{x}{u}$;
\item Sorting the top row in weakly increasing order, breaking
ties by sorting the bottom entries in \emph{increasing} order.
\end{itemize}
The definition of $T'$ is analogous to the Burge transpose $T$, the only
difference being that in case of ties we sort the bottom entries
in increasing order. Biwords where columns are sorted this way
play a central role in the RSK correspondence, and are often called
\emph{generalized permutations}~\cite{K}.
For instance, the $T'$-transpose of the biword
$(u,x)=(112224453,141233551)$ is
$$
\setlength\arraycolsep{1.5pt}
\begin{pmatrix}
1&1&2&2&2&4&4&5&3\\
1&4&1&2&3&3&5&5&1
\end{pmatrix}
\loongmapsto{\text{flip}}
\begin{pmatrix}
1&4&1&2&3&3&5&5&1\\
1&1&2&2&2&4&4&5&3
\end{pmatrix}
\loongmapsto{\text{sort}}
\begin{pmatrix}
1&1&1&2&3&3&4&5&5\\
1&2&3&2&2&4&1&4&5
\end{pmatrix}.
$$
Finally, we let $\psi(x)$ be the bottom row of the biword $(u,x)^{T'}$.
In our example, we have obtained $\psi(x)=123224145$.
Note that the bijection $\phi$ defined in the previous section maps
the same restricted growth function $123224135$ to
the modified ascent sequence $141233551$, and not to $x$.
That is, $\psi\neq\phi^{-1}$.

Next we prove that $\psi(x)$ is a restricted growth function
and $\psi$ is a bijection from $\Modasc(212)$ to $\RGF$.
Let us prove that $\psi(x)\in\RGF$ first. Consider the decomposition
of $(u,x)^{T'}$ obtained by splitting the biword according to the value
of the top entry; more explicitly, we have
$$
\setlength\arraycolsep{1.75pt}
\setcounter{MaxMatrixCols}{30}
\binom{u}{x}^{T'}=
\left(\begin{array}{cccc|cccc|cccc|cc|cccc}
1 & 1 & \dots & 1
& 2   & 2        & \dots & 2
& 3   & 3        & \dots & 3
& \dots &
& m   & m        & \dots & m\\
t_1 & 2 & \dots & \ell_1
& t_2 & \ell_1+1 & \dots & \ell_2
& t_3 & \ell_2+1 & \dots & \ell_3
& \dots &
& t_m & \ell_{m-1}+1 & \dots & \ell_m
\end{array}\right),
$$
where $t_i$ is the label assigned under $\psi$ to the leftmost copy
of $i$ in $x$, $m=\max(x)$, and $\ell_1,\dots,\ell_m$ are
nonnegative integers.
To prove that $x\in\RGF$, it suffices to show that $t_1=1$ and
$t_i\le\ell_{i-1}$ for each $i\ge 2$. By definition of $\psi$, we
have $t_1=1$. Furthermore, if $j$ is the index of the leftmost copy
of $i$ in $x$, then $t_i$ is equal to the label $u_{j-1}$ of $x_{j-1}$.
For $i=2$, we have $x_{j-1}<x_j=2$, thus $x_{j-1}=1$ and
$t_2\in\lbrace 1,2,\dots,\ell_1\rbrace$, i.e. $t_2\le \ell_1$.
A completely analogous argument shows that $t_i\le\ell_{i-1}$ for
each $i$. As a result, $\psi(x)$ is a restricted growth function.

Finally, we show that $\psi$ is bijective by defining its inverse
map $\psi^{-1}$ from $\RGF$ to $\Modasc(212)$.
Let $r=r_1\cdots r_n\in\RGF$. For $i\in[n]$, we define the label $y_i$
of $r_i$ as hinted by the decomposition of $(u,x)^{T'}$ considered
above. That is, we let $y_1=1$ and, for $i\ge 2$,
$$
y_i=
\begin{cases}
y_{i-1} & \text{if $r_{i}=\max(r_1\cdots r_{i-1})+1$};\\
y_{i-1}+1 & \text{if $r_i\le \max(r_1\cdots r_i)$}.
\end{cases}
$$
In other words, scanning $r$ from left to right, we repeat the same
label if the bottom entry is a new integer in $r$; otherwise, we
use a new label (equal to one plus the label used before).
We will now describe an iterative procedure to arrange the columns
$(r_i,y_i)$ in a biword, and the desired modified ascent sequence
$x=\psi^{-1}(r)$ will be obtained as the bottom row of such biword.
For $j=1,2,\dots,\max(y)$, at the $j$th step of the procedure
we shall arrange all the columns where $y_i=j$ as a consecutive
block, sorting them in increasing order with respect to the
top entry~$r_i$. Their position in the biword will be determined
by the smallest top label of such entries;
more precisely, if the smallest top label of the entries $y_i=j$
is equal to~$\ell$, then we insert all such columns immediately
after the rightmost column in the current biword where the
top entry is equal to~$\ell$.
At the first step, $j=1$, we simply arrange all the columns with
bottom entry~$1$ (sorting them in increasing order with respect to
the top entry). The fact that~$r$ is a restricted growth function
guarantees that such~$\ell$ is well defined at each next step
of the procedure.
%

Instead of proving that the map defined this way is the inverse map
of $\psi$, we wish to better clarify this construction with a
concrete example.
Consider the restricted growth function $r=123224145$ obtained
previously as $r=\psi(x)$, for $x=141233551$.
Letting $y=y_1\cdots y_n$, we have
$$
\setlength\arraycolsep{2pt}
\setcounter{MaxMatrixCols}{20}
\begin{array}{llcccc|ccc|cccc|ccc|ccc}
r= && 1 & 2 & 3 &&& 2 &&& 2 & 4 &&& 1 &&& 4 & 5;\\
y= && 1 & 1 & 1 &&& 2 &&& 3 & 3 &&& 4 &&& 5 & 5.\\
\end{array}
$$
We arrange the columns $(r_i,y_i)$ in a biword following the
iterative procedure described above. At the $j$th step, the newly
inserted columns are highlighted and $\ell_j$ denotes the smallest
top label of the entries $y_i=j$.
\begin{center}
\renewcommand{\arraystretch}{2.5}
\begin{tabular}{llp{0.5\textwidth}}
Step 1: &
$\phantom{\ell_1=0\;\longrightarrow\;\;}
\displaystyle{\binom{\mathbf{123}}{\mathbf{111}}}$;\\
Step 2: & $\ell_2=2\;\longrightarrow\;
\displaystyle{\binom{12\;\mathbf{2}\;3}{11\;\mathbf{2}\;1}}$;
\\
Step 3: &
$\ell_3=2\;\longrightarrow\;
\displaystyle{\binom{122\;\mathbf{24}\;3}{112\;\mathbf{33}\;1}}$;\\
Step 4: &
$\ell_4=1\;\longrightarrow\;
\displaystyle{\binom{1\;\mathbf{1}\;22243}{1\;\mathbf{4}\;12331}}$;\\
Step 5: &
$\ell_5=4\;\longrightarrow\;
\displaystyle{\binom{112224\;\mathbf{45}\;3}{141233\;\mathbf{55}\;1}}$.\\
\end{tabular}
\end{center}
In the end, the bottom row of the resulting biword is
$\psi^{-1}(r)=141233551=x$, as expected.

\begin{remark}
The map $\psi$ is in fact defined on the set $\Modasc$ of all
modified ascent sequences.
Quite remarkably, a numerical investigation suggests that the
restriction of $\psi$ to $\Modasc(2213)$ yields a bijection to
$\RGF$ too. A proof of this fact remains to be found.
When $\psi$ is extended to $\Modasc$, the smallest example of a collision
is given by
$$
\psi(12132)=\psi(12213)=12132.
$$
This is the only collision for sequences of length five (and indeed
$|\RGF_5|=|\Modasc_5|-1$).
Note that $12132$ and $12213$ are the shortest modified ascent sequences
containing $212$ and $2213$, respectively.
On the other hand, for instance due to the same example of collision, the
restriction of $\psi$ to $\Modasc(2231)$ and $\Modasc(2321)$
is not bijective.
\end{remark}

\section{Patterns $2213$ and $2231$}\label{section_2213_2231}

In Section~\ref{section_212}, we showed a recursive construction
of $\Modasc(212)$ that embodies Equation~\eqref{eq_stirl}
defining the Stirling numbers of the second kind.
In this section, we settle Conjecture~\ref{conj_ds} for the patterns
$2213$ and $2231$.
Our methods are similar in that we will once again construct
$\Modasc(y)$ recursively, for $y\in\lbrace 2213,2231\rbrace$.
This construction, however, leads to a different equation
$$
g_{n+1}(k)=k\sum_{j=k}^n g_{n}(j),
$$
where $g_{n}(k)$ is the number of modified ascent sequences in
$\Modasc_n(y)$ that contain $k$ copies of $1$.
To prove that the Bell numbers satisfy the same equation, we find an
analogous construction for $\RGF$. A bijection between the two sets
is obtained as a byproduct.

Let us start from $\RGF$. Throughout this section, we denote by
$\RGF_n(k)$ the set of restricted growth functions whose length of
the maximal, strictly increasing prefix is equal to $k$.
That is, for $k\le n-1$ we let
$$
\RGF_n(k)=\lbrace r\in\RGF_n:\;r_1\cdots r_k=12\cdots k,\ r_{k+1}\le k\rbrace,
$$
while $\RGF_n(n)$ is the singleton containing the sequence $12\cdots n$.
In terms of set partitions, sequences in $\RGF_n(k)$ correspond to set
partitions of $[n]$ where exactly $k$ blocks contain their own index as
an element; equivalently, where $k$ is the maximum integer such that
$1,2,\dots,k$ are contained in distinct blocks.
Let $h_n(k)=|\RGF_n(k)|$. We show that $h_{n+1}(n+1)=1$ and,
for $k=1,\dots,n$,
\begin{equation}\label{eq_H}
h_{n+1}(k)=k\sum_{j=k}^{n}h_{n}(j).
\end{equation}
Indeed, we just noted that $\RGF_{n+1}(n+1)$ is a singleton.
Furthermore, let $k\in[n]$ and let $r\in \RGF_{n}(j)$, for some $j\ge k$.
Then we have
$$
r=r_1\cdots r_jr_{j+1}\cdots r_n
\quad\text{with}\;
r_1\cdots r_j=1\cdots j
\;\text{and}\;r_{j+1}\le j.
$$
For $i\in [k]$, let $r(k,i)$ be the sequence obtained from $r$
by inserting a new entry equal to $i$ between $r_k$ and $r_{k+1}$;
that is, let
$$
\begin{matrix}
r(k,i) &=& r_1\cdots r_k & i & r_{k+1}\cdots r_j & r_{j+1}\cdots r_n\\
 &=& 1\cdots k     & i & k+1 \cdots j & r_{j+1}\cdots r_n.
\end{matrix}
$$
Clearly, $r(k,i)\in\RGF_{n+1}(k)$. On the other hand,
let $y\in\RGF_{n+1}(k)$, with $k<n+1$. Then $y=r(k,i)$,
where $i=y_{k+1}$ and $r$ is obtained by removing $y_{k+1}$
from $y$. Note that $r\in\RGF_n(j)$, for some $j\ge k$.
As a result, for $k<n+1$, we have the bijection
\begin{align*}
[k]\times\bigcup_{j=k}^{n}\RGF_n(j)&\longrightarrow \RGF_{n+1}(k)\\
(i,r)&\longmapsto r(k,i)
\end{align*}
and thus
$$
k\sum_{j=k}^n h_n(j)=h_{n+1}(k).
$$
The triangle $h_n(k)$ is recorded as A259691 in the OEIS~\cite{Sl}.
To the best of our knowledge, the combinatorial interpretation given
here appears to be new.

Next we turn our attention to $\Modasc(2213)$. For $n\ge 1$ and $k\in [n]$,
let
$$
G_n(k)=
\lbrace x\in\Modasc_n(2213):\;\text{$x$ contains $k$ copies of $1$}\rbrace
$$
and let $g_n(k)=|G_n(k)|$. Note that $g_{n}(n)=1$ since
the only sequence in $G_{n}(n)$ is $x=1\cdots 1$.
For $k<n+1$, we mimic what we did before and define a bijection
\begin{align*}
[k]\times\bigcup_{j=k}^{n}G_n(j)&\longrightarrow G_{n+1}(k)\\
(i,x)&\longmapsto x(k,i),
\end{align*}
proving that the coefficients $g_{n+1}(k)$ satisfy the same recurrence
as $h_{n+1}(k)$. Since the initial conditions are the same,
we have that $g_{n}(k)=h_{n}(k)$.

Let $x\in G_n(j)$, for some $j\ge k$, and let $i\in[k]$.
Define the sequence $x(k,i)$ as follows:

\begin{enumerate}
\item[(1)] Increase by one all the entries of $x$, except for the
$k$ leftmost copies of $1$.
\item[(2)] Insert a new entry equal to $2$ immediately to the right of
the $i$th copy of~$1$ (counting from left to right).
\item[(3)] In the special case where $i<k$ and the $k$th and $(k+1)$th
copies of $1$ are consecutive in position, move to the end of the sequence
the maximal string of consecutive $1$s that end with the $k$th copy.
\end{enumerate}

Next we show in Proposition~\ref{prop_2213} that $x(k,i)\in G_{n+1}(k)$.
Roughly speaking, the reason is that the map
$(i,x)\mapsto x(k,i)$ preserves the equality $\asctops(x)=\nub(x)$ and
does not create any new occurrence of $2213$.
The role of~(3) is to address the case where the $(k+1)$th copy of $1$
would become an ascent top that is not a leftmost copy.
The proof consists of a case by case analysis that requires a certain
amount of technicalities, some of which will be omitted for the sake
of readability.
The shortest example where~(3) plays a role is given by $x=1112\in G_4(3)$.
In this case, for $k=2$ and $i=1$, we have
$$
1\;1\;1\;2\;\loongmapsto{(1)}\;
1\;1\;2\;3\;\loongmapsto{(2)}\;
1\;\underline{2}\;\boxed{1}\;2\;3\;\loongmapsto{(3)}\;
1\;\underline{2}\;2\;3\;\boxed{1}\;=x(2,1),
$$
where the underlined entry is the newly inserted $2$.
Note that, without moving the boxed $1$ (i.e. the maximal string of
consecutive $1$s that end with the $k$th copy of $1$) at the end of
the sequence, the $(k+1)$th copy of $1$ would be an ascent top
in $x(2,1)$, but not a leftmost copy; consequently, the resulting
sequence $12123$ would not be a modified ascent sequence.
Two more instances of this construction will be illustrated later
in Example~\ref{example_2213} and Example~\ref{example_2213_b}.

\begin{proposition}\label{prop_2213}
Let $x\in G_n(j)$, with $k\le j \le n$, and let $i\in[k]$.
Then $x(k,i)\in G_{n+1}(k)$.
\end{proposition}
\begin{proof}
Let $y=x(k,i)$. By definition, $y$ contains $k$ copies of $1$.
We have to show that $y\in\Modasc$ and $y$ avoids $2213$.
We distinguish two cases depending on whether or not
the definition of $y=x(k,i)$ falls under the special case $(3)$.

\begin{itemize}
\item Suppose that we are not in the special case $(3)$; that is,
$y$ is obtained from $x$ by increasing by one all its entries,
except for the leftmost $k$ copies of $1$, and inserting a $2$
immediately after the $i$th copy of $1$.
More explicitly, in this case we have
$$
\renewcommand*{\arraystretch}{1.5}
\begin{matrix}
x= & 1^{(1)}B_1 & \cdots & 1^{(i)}B_i & \cdots &
1^{(k)}B_k\ 1^{(k+1)}B_{k+1} & \cdots & 1^{(j)}B_j,\\
y= & 1^{(1)}\bar{B}_1 & \cdots & 1^{(i)}2\bar{B}_i & \cdots &
1^{(k)}\bar{B}_k\ \bar{1}^{(k+1)}\bar{B}_{k+1} & \cdots & \bar{1}^{(j)}\bar{B}_j,
\end{matrix}
$$
where, for $\ell=1,\dots,j$, $1^{(\ell)}$ denotes the $\ell$th copy of $1$,
$B_{\ell}$ denotes the block of entries between $1^{(\ell)}$ and $1^{(\ell+1)}$,
and a bar marks entries and blocks that are increased by one in $y$.
We prove that $y\in\Modasc$ by showing that $\asctops(y)=\nub(y)$.
Note that $\asctops(x)=\nub(x)$ since $x\in\Modasc$.
The relative order of the entries of $y$ is almost the same as in $x$,
with the exception of the newly inserted~$2$ and the increased copies
of $1$. In particular, the newly inserted $2$ is both an ascent top
and a leftmost copy in $y$.
Furthermore, if the block $B_i$ is not empty, then its leftmost element
is an ascent top both in $x$ and in $y$, since all the elements in
$\bar{B}_i$ are greater than $2$. Similarly, it belongs to $\nub(x)$,
since $\nub(x)=\asctops(x)$, and to $\nub(y)$ as well,
once again due to the fact that the relative order of the entries
of $y$ that are greater than $2$ is the same as it was in $x$.
The case when $B_i$ is empty can be addressed similarly.
Finally, the only other entry that could potentially become a new ascent
top in $y$ is $\bar{1}^{(k+1)}$. This happens if and only if $B_k$ is empty,
that is, if $1^{(k)}$ and $1^{(k+1)}$ are consecutive in positions.
Now, if $i=k$, then the entries
$$
1^{(k)}1^{(k+1)}
\quad\text{are mapped to}\quad
1^{(k)}2\bar{1}^{(k+1)}=122
$$
and once again the newly inserted $2$ belongs to both $\asctops(y)$ and
$\nub(y)$. If instead $i<k$, then we fall under the special case $(3)$,
which we consider below.
As a result of the above discussion, the equality $\asctops(x)=\nub(x)$
is preserved by the construction $(i,x)\mapsto y$; hence $\asctops(y)=\nub(y)$
and $y\in\Modasc_{n+1}$, as wanted.\\
Let us show that $y$ avoids $2213$ next.
Note that $x$ avoids $2213$ by our assumptions.
Since the relative order of elements of $y$ that are greater than $2$
is the same as it was in $x$, an eventual occurrence
$y_{\ell_1}y_{\ell_2}y_{\ell_3}y_{\ell_4}$ of $2213$ in $y$
should satisfy $y_{\ell_3}=1$ and $y_{\ell_1}=y_{\ell_2}=2$;
otherwise, we would have $y_{\ell_4}\ge y_{\ell_1}=y_{\ell_2}\ge 3$
and the same four entries would form an occurrence of $2213$ in $x$,
which is impossible. On the other hand, there is only one entry equal to $2$
that precedes the rightmost copy of $1$ in $y$, and thus the case
$y_{\ell_3}=1$ and $y_{\ell_1}=y_{\ell_2}=2$ is impossible too.

\item Let us now take care of the special case (3) where $i<k$ and
$1^{(k)}$ and $1^{(k+1)}$ are consecutive in position. Let us write
$$
\begin{matrix}
x=& 1^{(1)}B_1 & \cdots & 1^{(i)}B_i & \cdots &
\boxed{1^{(\ell)}1^{(\ell+1)}\cdots 1^{(k)}} &
1^{(k+1)}B_{k+1} & \cdots & 1^{(j)}B_j,
\end{matrix}
$$
where the box contains the maximal string of consecutive $1$s
that end with $1^{(k)}$ (for some $\ell\ge 1$) and the rest is defined as
in the previous case. Then
$$
\renewcommand*{\arraystretch}{1.5}
\begin{matrix}
x=& 1^{(1)}B_1 & \cdots & 1^{(i)}B_i & \cdots
& \boxed{1^{(\ell)}\cdots 1^{(k)}}
& 1^{(k+1)}B_{k+1} \cdots 1^{(j)}B_j,\\
y=& 1^{(1)}\bar{B}_1 & \cdots & 1^{(i)}2\bar{B}_i & \cdots
& \bar{1}^{(k+1)}\bar{B}_{k+1}\cdots\bar{1}^{(j)}\bar{B}_j
& \boxed{1^{(\ell)}\cdots 1^{(k)}}\;,
\end{matrix}
$$
where in $y$ the boxed $1$s have been moved at the end under the effect
of $(3)$. It is now easy to verify that, similarly to the previous case,
the equality $\asctops(y)=\nub(y)$ holds, and we omit the details.
We just note that, by moving the boxed $1$s at the end of the sequence,
we have removed the newly created ascent $1^{(k)}\bar{1}^{(k+1)}$.
Otherwise, the entry $\bar{1}^{(k+1)}$ would be an ascent top, but not a
leftmost copy since $\bar{1}^{(k+1)}=2$ and we inserted a new entry equal
to $2$ immediately after $1^{(i)}$, with $i<k$.
The proof that $y$ avoids $2213$ is identical to the previous case, with the
addition that moving the boxed $1$s at the end of the sequence cannot create
an occurrence of $2213$.
\end{itemize}
\end{proof}

To see that the map $(i,x)\mapsto x(k,i)$ is a bijection from
$[k]\times \bigcup_{j=k}^n G_n(j)$ to $G_{n+1}(k)$, let us
define its inverse map.
Let $y=y_1\cdots y_{n+1}\in G_{n+1}(k)$. We determine $i,j,k$ and $x$
such that $x\in G_n(j)$ and $y=x(k,i)$. Let
\begin{align*}
i&=|\lbrace \ell:\ y_{\ell}=1\text{ and } y_{s}\neq 2\ \forall s<\ell\rbrace|,\\
j&=|\lbrace \ell: y_{\ell}\le 2\rbrace|-1,\\
k&=|\lbrace \ell: y_{\ell}=1\rbrace\rbrace|;
\end{align*}
that is, $i$ is equal to the number of $1$s preceding the leftmost $2$;
$j$ is equal to the number of copies of $1$ and $2$, minus one;
and $k$ is equal to number of copies of $1$s.
Let $x$ be the string obtained from $y$ as follows:
\begin{itemize}
\item[$(a)$] Remove the leftmost copy of $2$; note that the leftmost
copy of $2$ is an ascent top preceded by the $i$th copy of $1$.
\item[$(b)$] In the special case where $y_{n+1}=1$ and there is at least
one copy of $2$ left (after $(a)$ has been applied),
move the maximal string of consecutive $1$s containing $y_{n+1}$
immediately before the leftmost copy of $2$.
\item[$(c)$] Decrease by one each entry that is strictly greater than $1$.
\end{itemize}
Clearly, steps $(a)$, $(b)$ and $(c)$ mirror, respectively, steps $(2)$,
$(3)$ and $(1)$ in the definition of the map $(x,i)\mapsto x(k,i)$.
Note that $(b)$ is applied if and only $x\mapsto x(k,i)=y$ falls under
the special case $(3)$. Indeed, if $(3)$ is not applied, then in $y$ all
the copies of $1$ precede the second copy of $2$, i.e. $\bar{1}^{(k+1)}$.
Thus, when the leftmost copy of $2$ is removed from $y$ at step $(a)$,
the last entry $y_{n+1}=1$ is preceded by another copy of $2$ if and only
if the final string of consecutive $1$s was moved under the effect of $(3)$.
In the end, we have $y=x(k,i)$, and the map defined above is the
inverse map of $(x,i)\mapsto x(k,i)$.

Two examples of this construction are illustrated below.

\begin{example}\label{example_2213}
Let $y=12613224532$. Note that $y\in G_{11}(2)$. Here we have
$$
i=1,\quad
j=5,\quad
k=2
$$
and $x$ is obtained as
$$
y=1\underline{2}613224532\;\loongmapsto{$(a)$}\;
1613224532\;\loongmapsto{$(c)$}\;
1512113421=x,
$$
where the leftmost copy of $2$ in $y$ is underlined.
Note that $(b)$ is not applied since $y$ does not end
with $1$. It is now easy to check that $y=x(2,1)$:
$$
x=1512113421\;\loongmapsto{(1)}\;
1613224532\;\loongmapsto{(2)}\;
1\underline{2}613224532=y.
$$
\end{example}

\begin{example}\label{example_2213_b}
Let $y=131551242111$. We have $y\in G_{12}(6)$ and
$$
i=3,\quad
j=7,\quad
k=6.
$$
Here we apply $(a)$, $(b)$ and $(c)$ to obtain
\begin{align*}
y=131551\underline{2}42111\;
&\loongmapsto{$(a)$}\;13155142\;\boxed{111}\\
&\loongmapsto{$(b)$}\;1315514\;\boxed{111}\;2\\
&\loongmapsto{$(c)$}\;12144131111=x.
\end{align*}
In this case, $(c)$ is applied since $y$ ends with $1$ and there is at
least one copy of $2$ in the sequence resulting from $(b)$.
Finally, we have $y=x(6,3)$:
\begin{align*}
x=12144131111\;
&\loongmapsto{(1)}\;13155141112\\
&\loongmapsto{(2)}\;131551\underline{2}4\;\boxed{111}\;2\\
&\loongmapsto{(3)}\;131551242\;\boxed{111}\;=y.
\end{align*}
\end{example}

In the previous part of this section, we have proved
Equation~\eqref{eq_H} by showing a recursive construction
where each object of $\RGF_n(j)$, $j=k,k+1,\dots,n$, gives rise to~$k$ objects in $\RGF_{n+1}(k)$. In a similar fashion,
we described a recursive construction of $G_n(k)=\lbrace x\in\Modasc_n(2213):\text{$x$ contains $k$ copies of $1$}\rbrace$
which leads to the analogous equation
$$
g_{n+1}(k)=k\sum_{j=k}^n g_{n}(j),
$$
where $g_{n}(j)=|G_{n}(j)|$. The next corollary follows
immediately from these two results.

\begin{corollary}
For $n\ge 0$, the number of $2213$-avoiding modified ascent sequences
of length $n$ is equal to the $n$th Bell number.
\end{corollary}
\begin{proof}
Let $k\in [n]$. Due to what proved so far in this section, the
number of restricted growth functions in $\RGF_n(k)$ equals the number
of $2213$-avoiding modified ascent sequences of length $n$ that
contain $k$ copies of $1$.
By summing over $k$, we obtain
$$
|\Modasc_n(2213)|=\sum_{k=1}^n |G_n(k)|=
\sum_{k=1}^n |\RGF_n(k)|=|\RGF_n|,
$$
which is equal to the $n$th Bell number.
\end{proof}

Two methods to construct inductively $\RGF$ and $\Modasc(2213)$
are provided, respectively, by the maps
$$
r\mapsto r(k,i)
\quand
x\mapsto x(k,i).
$$
A bijection $\phi$ from $\RGF_n$ to $\Modasc_n(2213)$ is obtained
accordingly by letting
$$
\phi(12\cdots n)=11\cdots 1
\quand
\phi\bigl(r(k,i)\bigr)=\phi(r)(k,i).
$$
In other words, we let $\phi$ map the only sequence $r=12\cdots n$
in $\RGF_{n}(n)$ to the only sequence $x=11\cdots 1$ in $G_n(n)$;
and, if $\phi(r)=x$ is defined inductively and $s=r(k,i)$, then we
let $\phi(s)=x(k,i)$.
Note that $\max(12\cdots n)=n+1-\max(11\cdots 1)$; and, if we assume
that $\max(r)=n+1-\max(x)$, then
$$
\max\bigl(r(k,i)\bigr)=\max(r)=n+1-\max(x)=n+2-\max\bigl(x(k,i)\bigr),
$$
where we used the fact that
$$
\max\bigl(r(k,i)\bigr)=\max(r)
\quand
\max\bigl(x(k,i)\bigr)=\max(x)+1.
$$
Therefore, for each $r\in\RGF_n$, we have
$$
\max(r)+\max\bigl(\phi(r)\bigr)=n+1.
$$
As observed at the end of Section~\ref{section_212} for the pattern $212$,
a consequence of this fact is that the distribution of
$|\asctops(x)|=\max(x)$ on $\Modasc(2213)$ is given by the reverse of the
Stirling numbers of the second kind. We have thus settled Conjecture~\ref{conj_ds}
for the pattern $2213$.

The pattern $2231$ can be solved by slightly tweaking the argument used
for $2213$. The same recursive construction works, except for the fact
that the special rule $(3)$ in the definition of the map
$x\mapsto x(k,i)$ must be replaced with:
\begin{itemize}
\item[$(3')$] In the special case where $i<k$ and the $k$th and $(k+1)$th
copies of $1$ are in consecutive positions in $x$, move
the maximal string of consecutive $1$s that end with the $k$th copy
\emph{immediately after the maximal string of consecutive $2$s that starts
with $\bar{1}^{(k+1)}$}.
\end{itemize}

Recall that by applying $(3)$ to the sequence $1112$ we obtained
$$
1\;1\;1\;2\;\loongmapsto{(1)}\;
1\;1\;2\;3\;\loongmapsto{(2)}\;
1\;\underline{2}\;\boxed{1}\;2\;3\;\loongmapsto{(3)}\;
1\;\underline{2}\;2\;3\;\boxed{1},
$$
which avoids $2213$ (but contains $2231$).
If instead we apply $(3')$, we obtain
$$
1\;1\;1\;2\;\loongmapsto{(1)}\;
1\;1\;2\;3\;\loongmapsto{(2)}\;
1\;\underline{2}\;\boxed{1}\;2\;3\;\loongmapsto{(3)}\;
1\;\underline{2}\;2\;\boxed{1}\;3,
$$
which avoids $2231$ (but contains $2213$).
Roughly speaking, both the constructions used for $2231$ and $2213$
preserve the equality $\asctops(x)=\nub(x)$.
The special rules $(3)$ and $(3')$ are necessary
in order to address the case where $1^{(k)}$ and $1^{(k+1)}$ are
in consecutive positions, and this would result in a strict ascent
$1^{(k)}\bar{1}^{(k+1)}$ where $\bar{1}^{(k+1)}$ is not a leftmost copy.
More specifically, this is fixed by moving the box of consecutive $1$s
ending with $1^{(k)}$ somewhere else (to the right) in the sequence.
In order to avoid $2213$, we move the box as far as possible at the
end of the sequence; if instead we want to avoid $2231$, we move the box
as close as possible to the original spot.

\begin{theorem}\label{proof_2213_2231}
Let $y\in\lbrace 2213,2231\rbrace$. Then $|\Modasc_n(y)|$
is equal to the $n$th Bell number. Furthermore, the number of
modified ascent sequences in $\Modasc_n(y)$ whose maximum value
is equal to $m$ is given by the $(n,n-m+1)$th Stirling
number of the second kind.
\end{theorem}

\section{Fishburn permutations}\label{section_fish_perms}

Recall from Section~\ref{section_intro_fish} the 
bijection $\gamma:\Modasc\to\F$ defined by
$$
\begin{pmatrix}
\identity\\
x
\end{pmatrix}^T=
\begin{pmatrix}
\sort(x)\\
\gamma(x)
\end{pmatrix},
$$
where $\identity$ is the identity permutation, $x\in\Modasc$,
$\sort(x)$ is obtained by sorting the entries of $x$ in weakly
increasing order, and $T$ denotes the Burge transpose.
Recall also that the biwords $(\identity,x)$ and
$\bigl(\sort(x),\gamma(x)\bigr)$ are Burge words; that is, the
descent set of the top row is a subset of the descent set of the
bottom row, and this property is preserved by $T$.

In this section,  we use bivincular patterns to characterize the set
$\gamma\bigl(\Modasc(y)\bigr)$ of Fishburn permutations
corresponding to $\Modasc(y)$, for $y\in\lbrace 212,2213,2321\rbrace$.
Namely, we show that
$$
\gamma\bigl(\Modasc(212)\bigr)=\F(\alpha),\quad
\gamma\bigl(\Modasc(2213)\bigr)=\F(\beta_1,\beta_2),\quad
\gamma\bigl(\Modasc(2321)\bigr)=\F(\delta_1,\delta_2),
$$
where the bivincular patterns $\alpha$, $\beta_1$, $\beta_2$,
$\delta_1$, $\delta_2$ are depicted---as mesh patterns---in Figure~\ref{figure_patterns}.
As usual, we shall assume the same picture as their definition.
For instance, a more extensive definition of $\alpha$ would be
$$
\alpha=\bigl(2413,
\{(2,k): k=0,\dots,4\}\cup
\{(k,3): k=0,\dots,4\}\bigr).
$$
We could not find a suitable description of
$\gamma\bigl(\Modasc(2231)\bigr)$ in terms of pattern avoidance.

\begin{figure}
\centering
\begin{tabular}{c}
$\alpha\;=\;\begin{tikzpicture}[scale=0.40, baseline=20pt]
\fill[NE-lines] (2,0) rectangle (3,5);
\fill[NE-lines] (0,3) rectangle (5,4);
\draw [semithick] (0.001,0.001) grid (4.999,4.999);
\filldraw (1,2) circle (6pt);
\filldraw (2,4) circle (6pt);
\filldraw (3,1) circle (6pt);
\filldraw (4,3) circle (6pt);
\end{tikzpicture}$
\qquad
$\beta_1\;=\;\begin{tikzpicture}[scale=0.40, baseline=20pt]
\fill[NE-lines] (0,2) rectangle (6,3);
\draw [semithick] (0.001,0.001) grid (5.999,5.999);
\filldraw (1,4) circle (6pt);
\filldraw (2,1) circle (6pt);
\filldraw (3,3) circle (6pt);
\filldraw (4,2) circle (6pt);
\filldraw (5,5) circle (6pt);
\end{tikzpicture}$
\qquad
$\beta_2\;=\;\begin{tikzpicture}[scale=0.40, baseline=20pt]
\fill[NE-lines] (0,2) rectangle (6,3);
\draw [semithick] (0.001,0.001) grid (5.999,5.999);
\filldraw (1,4) circle (6pt);
\filldraw (2,1) circle (6pt);
\filldraw (3,3) circle (6pt);
\filldraw (4,5) circle (6pt);
\filldraw (5,2) circle (6pt);
\end{tikzpicture}$
\\
$\delta_1\;=\;\begin{tikzpicture}[scale=0.40, baseline=20pt]
\fill[NE-lines] (0,3) rectangle (6,4);
\draw [semithick] (0.001,0.001) grid (5.999,5.999);
\filldraw (1,5) circle (6pt);
\filldraw (2,1) circle (6pt);
\filldraw (3,4) circle (6pt);
\filldraw (4,2) circle (6pt);
\filldraw (5,3) circle (6pt);
\end{tikzpicture}$
\qquad
$\delta_2\;=\;\begin{tikzpicture}[scale=0.40, baseline=20pt]
\fill[NE-lines] (0,3) rectangle (6,4);
\draw [semithick] (0.001,0.001) grid (5.999,5.999);
\filldraw (1,5) circle (6pt);
\filldraw (2,2) circle (6pt);
\filldraw (3,4) circle (6pt);
\filldraw (4,1) circle (6pt);
\filldraw (5,3) circle (6pt);
\end{tikzpicture}$
\end{tabular}
\caption{Mesh (bivincular) patterns such that
$\gamma\bigl((\Modasc(212)\bigr)=\F(\alpha)$,
$\gamma\bigl((\Modasc(2213)\bigr)=\F(\beta_1,\beta_2)$,
and $\gamma\bigl((\Modasc(2321)\bigr)=\F(\delta_1,\delta_2)$.
}\label{figure_patterns}
\end{figure}

Let us start from the pattern $2321$.

\begin{proposition}\label{fish_2321}
We have
$$
\gamma\bigl(\Modasc(2321)\bigr)=\F(\delta_1,\delta_2).
$$
\end{proposition}
\begin{proof}
Let $x\in\Modasc_n(2321)$ and let $p=\gamma(x)$ be the
corresponding Fishburn permutation. We show that
$$
x\text{ contains }2321
\quad\text{if and only if}\quad
p\text{ contains }\delta_1\text{ or }\delta_2.
$$
Suppose that $x$ contains an occurrence $x_{i_1}x_{i_2}x_{i_3}x_{i_4}$
of $2321$. Let $j$ be the index of the leftmost copy of
$x_{i_4}$ in $x$, i.e. such that $(j,x_j)\in\nub(x)$ and $x_j=x_{i_4}$.
We start by showing that we can assume $i_3=i_2+1$ without losing
generality.
Indeed, we have $x_{i_1}=x_{i_3}$ and thus, since $\asctops(x)=\nub(x)$,
$x_{i_3}$ is not an ascent top, i.e. $x_{i_3-1}\ge x_{i_3}$.
Now, if $x_{i_3-1}>x_{i_3}$, then $x_{i_1}x_{i_3-1}x_{i_3}x_{i_4}$
is an occurrence of $2321$ where the second and third element are
in consecutive position, as wanted. Otherwise, if $x_{i_3-1}=x_{i_3}$,
then we can replace $x_{i_3}$ with $x_{i_3-1}$ and repeat the same
argument until we fall back in the previous case.
Similarly, we can assume $j<i_2$. Indeed, we have $x_{j-1}<x_j$
since $(j,x_j)\in\nub(x)=\asctops(x)$.
If $j>i_2$ (and thus also $j>i_3=i_2+1$), then
$x_{i_1}x_{i_2}x_{i_3}x_{j-1}$ is an occurrence of $2321$ and we can
once again go on until we eventually fall in the case $j<i_2$.
In the end, due to the assumptions $i_3=i_2+1$ and $j<i_2$,
we have either
$$
x=\cdots x_{j}\cdots x_{i_1}\cdots x_{i_2}x_{i_3}\cdots x_{i_4}
\quad\text{or}\quad
x=\cdots x_{i_1}\cdots x_{j}\cdots x_{i_2}x_{i_3}\cdots x_{i_4},
$$
depending on whether $j<i_1$ or $j>i_1$. Now, let us consider more
in details the equation $(\identity_n,x)^T=(\sort(x),p)$.
If $j<i_1$, then we have
$$
\setlength\arraycolsep{1pt}
\begin{pmatrix}
\cdots & {j}   &\cdots & {i_1}  &\cdots & {i_2}{i_3}    &\cdots & {i_4} & \cdots\\
\cdots & x_{j} &\cdots & x_{i_1}&\cdots & x_{i_2}x_{i_3}&\cdots & x_{i_4}& \cdots
\end{pmatrix}^T=
\begin{pmatrix}
\cdots & x_{i_4}&\cdots& x_{j}&\cdots& x_{i_3}&\cdots&x_{i_1}&\cdots& x_{i_2}& \cdots\\
\cdots & i_4    &\cdots& j    &\cdots& i_3    &\cdots&i_1    &\cdots& i_2& \cdots
\end{pmatrix}
$$
and $i_4,j,i_3,i_1,i_2$ is an occurrence of $\delta_1$ in $p$. Indeed,
the underlying pattern is $51423$ and $i_3=i_2+1$.
On the other hand, suppose that $j>i_1$. Then
$$
\setlength\arraycolsep{1pt}
\begin{pmatrix}
\cdots & {i_1}   &\cdots & {j}  &\cdots & {i_2}{i_3}    &\cdots & {i_4}& \cdots\\
\cdots & x_{i_1} &\cdots & x_{j}&\cdots & x_{i_2}x_{i_3}&\cdots & x_{i_4}& \cdots
\end{pmatrix}^T=
\begin{pmatrix}
\cdots & x_{i_4}&\cdots& x_{j}&\cdots& x_{i_3}&\cdots&x_{i_1}&\cdots& x_{i_2}& \cdots\\
\cdots & i_4    &\cdots& j    &\cdots& i_3    &\cdots&i_1    &\cdots& i_2& \cdots
\end{pmatrix}.
$$
In this case, the pattern underlying $i_4,j,i_3,i_1,i_2$ is $52413$
and, since $i_3=i_2+1$, $i_4,j,i_3,i_1,i_2$ is an occurrence
of $\delta_2$.
We have thus proved that if $x$ contains $2321$, then $p$ contains
$\delta_1$ or $\delta_2$, as desired.

To prove the opposite direction, we show that, if $p$ contains
$\delta_1$ or $\delta_2$, then $x$ contains $2321$.
Suppose initially that $p$ contains $\delta _1$.
Let $p_{i_1}p_{i_2}p_{i_3}p_{i_4}p_{i_5}$ be an occurrence
of $\delta_1$ in $p$; that is, an occurrence of $51423$ where
$p_{i_3}=p_{i_5}+1$. Let
$$
\setlength\arraycolsep{1.5pt}
\setcounter{MaxMatrixCols}{15}
\binom{\sort(x)}{p}=
\begin{pmatrix}
\cdots & \ell_1 &\cdots& \ell_2&\cdots& \ell_3
&\cdots&\ell_4&\cdots& \ell_5 & \cdots\\
\cdots & p_{i_1}&\cdots& p_{i_2}&\cdots& p_{i_3}
&\cdots& p_{i_4}&\cdots& p_{i_5}&\cdots
\end{pmatrix}.
$$
Since $(\sort(x),p)$ is a Burge word, we have
$\Des\bigl(\sort(x)\bigr)\subseteq\Des(p)$. Hence it must be
$$
\ell_2<\ell_3\;\text{ since }\;p_{i_2}<p_{i_3}
\quand
\ell_4<\ell_5\;\text{ since }\;p_{i_4}<p_{i_5}.
$$
Furthermore, the Burge transpose acts as
\begin{align*}
\setlength\arraycolsep{1.5pt}
\binom{\sort(x)}{p}^T
&=\begin{pmatrix}
\cdots & \ell_1 &\cdots& \ell_2&\cdots& \ell_3
&\cdots&\ell_4&\cdots& \ell_5 & \cdots\\
\cdots & p_{i_1}&\cdots& p_{i_2}&\cdots& p_{i_3}
&\cdots& p_{i_4}&\cdots& p_{i_5}&\cdots
\end{pmatrix}^T\\
&=\begin{pmatrix}
\cdots & p_{i_2}&\cdots& p_{i_4}&\cdots& p_{i_5}\;
p_{i_3}&\cdots& p_{i_1}&\cdots\\
\cdots & \ell_2 &\cdots& \ell_4&\cdots& \ell_5\;
\ell_3&\cdots& \ell_1 & \cdots
\end{pmatrix}
=\binom{\identity_n}{x}.
\end{align*}
Note that $\ell_5>\ell_3>\ell_1$ and $\ell_5$ and $\ell_3$ are in consecutive
positions since $p_{i_3}=p_{i_5}+1$. In particular, such entry $\ell_3$
is not an ascent top in $x$. Thus the leftmost copy, say $x_j$,
of $\ell_3$ precedes $\ell_5$ in $x$ (more precisely, it precedes the
column $(p_{i_5},\ell_5)$). Therefore, the entry $x_j$ form an occurrence
of $2321$ together with the bottom entries $\ell_5$, $\ell_3$ and $\ell_1$
in the columns $(p_{i_5},\ell_5)$, $(p_{i_3},\ell_3)$ and $(p_{i_1},\ell_1)$.
We have thus proved that, if $p$ contains $\delta_1$, then $x$ contains
$2321$.
Finally, suppose that $p$ contains an occurrence
$p_{i_1}p_{i_2}p_{i_3}p_{i_4}p_{i_5}$ of $\delta_2$. The proof that $x$
contains $2321$ is identical to the previous case, the only
difference being that the positions of the columns $(p_{i_2},i_2)$
and $(p_{i_4},i_4)$ is exchanged since the classical pattern underlying 
$\delta_2$ is $52413$ (instead of $51423$).
\end{proof}

The two remaining patterns $212$ and $2213$ can be solved in a similar
fashion. Below we just sketch the corresponding proofs, leaving
the details to the reader.

\begin{proposition}\label{fish_2213}
We have
$$
\gamma\bigl(\Modasc(2213)\bigr)=\F(\beta_1,\beta_2).
$$
\end{proposition}
\begin{proof}
Let $x\in\Modasc_n(2213)$ and let $p=\gamma(x)$.
We show that
$$
x\text{ contains }2213
\quad\text{if and only if}\quad
p\text{ contains }\beta_1\text{ or }\beta_2.
$$
Let $x_{i_1}x_{i_2}x_{i_3}x_{i_4}$ be an occurrence of $2213$ in $x$.
Let $x_j$ be the leftmost copy of $x_{i_3}$ in $x$.
Note that
$$
x_j=x_{i_3}<x_{i_1}=x_{i_2}<x_{i_4}.
$$
It is not hard to show that one can assume $j<i_2$ without losing
generality.
If $i_2=i_1+1$, then $i_3,j,i_2,i_1,i_4$ is an occurrence of $\beta_1$
in $p$. Otherwise, consider the entry $x_{i_2-1}$ immediately to the left
of $x_{i_2}$. Since $x_{i_2}\notin\nub(x)=\asctops(x)$, it must
be $x_{i_2-1}\ge x_{i_2}$ (and $j\neq i_2-1$). If $x_{i_2-1}=x_{i_2}$, replace $x_{i_2}$
with $x_{i_2-1}$ and repeat the same argument. Otherwise, suppose that
$x_{i_2-1}>x_{i_2}$. If $x_{i_2-1}<x_{i_4}$, then
\begin{align*}
\setlength\arraycolsep{1pt}
\binom{\identity}{x}^T&=
\begin{pmatrix}
\cdots & {j}   &\cdots & {i_2-1}\;{i_2}  &\cdots & {i_3}    &\cdots & {i_4}&\cdots\\
\cdots & x_{j} &\cdots & x_{i_2-1}\;x_{i_2}&\cdots & x_{i_3}&\cdots & x_{i_4}&\cdots
\end{pmatrix}^T\\
&=\begin{pmatrix}
\cdots & x_{i_3}&\cdots& x_{j}&\cdots& x_{i_2}&\cdots&x_{i_2-1}&\cdots& x_{i_4}&\cdots\\
\cdots & i_3    &\cdots& j    &\cdots& i_2    &\cdots&i_2-1    &\cdots& i_4&\cdots
\end{pmatrix}
=\binom{\sort(x)}{p}
\end{align*}
and $i_3,j,i_2,i_2-1,i_4$ is an occurrence of $\beta_1$
(note that $i_3,j,i_2,i_2-1,i_4\simeq 41325$). Similarly, if
$x_{i_2-1}\ge x_{i_4}$, then
\begin{align*}
\setlength\arraycolsep{1pt}
\binom{\identity}{x}^T&=
\begin{pmatrix}
\cdots & {j}   &\cdots & {i_2-1}\;{i_2}  &\cdots & {i_3}    &\cdots & {i_4}&\cdots\\
\cdots & x_{j} &\cdots & x_{i_2-1}\;x_{i_2}&\cdots & x_{i_3}&\cdots & x_{i_4}&\cdots
\end{pmatrix}^T\\
&=\begin{pmatrix}
\cdots & x_{i_3}&\cdots& x_{j}&\cdots& x_{i_2}&\cdots&x_{i_4}&\cdots& x_{i_2-1}&\cdots\\
\cdots & i_3    &\cdots& j    &\cdots& i_2    &\cdots&i_4    &\cdots& i_2-1&\cdots
\end{pmatrix}
=\binom{\sort(x)}{p}.
\end{align*}
and $i_3,j,i_2,i_4,i_2-1$ is an occurrence of $\beta_2$
(here $i_3,j,i_2,i_2-1,i_4\simeq 41352$).

Let us now take care of the other direction. Suppose that $p$ contains 
an occurrence $p_{i_1}p_{i_2}p_{i_3}p_{i_4}p_{i_5}$ of $\beta_1$;
that is, $p_{i_1}p_{i_2}p_{i_3}p_{i_4}p_{i_5}\simeq 41325$ and
$p_{i_3}=p_{i_4}+1$. The Burge transpose acts on
$(\sort(x),p)^T=(\identity,x)$ by mapping the columns
$$
\begin{pmatrix}
\setlength\arraycolsep{1.5pt}
\ell_1  & \ell_2  & \ell_3  & \ell_4  & \ell_5\\
p_{i_1} & p_{i_2} & p_{i_3} & p_{i_4} & p_{i_5}
\end{pmatrix}
\quad\text{to}\quad
\begin{pmatrix}
p_{i_2} & p_{i_4} & p_{i_3} & p_{i_1} & p_{i_5}\\
\ell_2  & \ell_4  & \ell_3  & \ell_1  & \ell_5
\end{pmatrix},
$$
where $\ell_4$ and $\ell_3$ are in consecutive positions in $p$ since
$p_{i_3}=p_{i_4}+1$. Due to the equality
$\Des\bigl(\sort(x)\bigr)\subseteq\Des(p)$, we also have
$$
\ell_2<\ell_3\;\text{ since }\;p_{i_2}<p_{i_3}
\quand
\ell_4<\ell_5\;\text{ since }\;p_{i_4}<p_{i_5}.
$$
Now, if $\ell_4=\ell_3$, then $\ell_4,\ell_3,\ell_1,\ell_5$ is an occurrence
of $2213$ in $x$. Otherwise, let $\ell_4>\ell_3$. Then $\ell_3$ is not an
ascent top in $x$ and $t,\ell_3,\ell_1,\ell_5$ is an occurrence of $2213$,
where $t$ is the leftmost copy of $\ell_3$ in $x$.
In a similar fashion, it is easy to see that, given an occurrence
$p_{i_1}p_{i_2}p_{i_3}p_{i_4}p_{i_5}$ of $\beta_2$,
the Burge transpose maps the entries $p_{i_1},p_{i_3},p_{i_4}$ to
an occurrence of $213$ in $x$ where the entry that plays the role
of $2$ is not an ascent top---and thus not a leftmost copy.
The desired occurrence of $2213$ is obtained immediately by adding
the leftmost copy of the $2$ to these three entries.
\end{proof}

\begin{proposition}\label{fish_212}
We have
$$
\gamma\bigl(\Modasc(212)\bigr)=\F(\alpha).
$$
\end{proposition}
\begin{proof}
Let $x\in\Modasc(212)$ and let $p=\gamma(x)$.
Let $x_{i_1}x_{i_2}x_{i_3}$ be an occurrence of $212$ in $x$.
Without losing generality, we can assume that
\begin{itemize}
\item[$(i)$] $x_j\neq x_{i_1}$ for each $i_1<j<i_3$;
\item[$(ii)$] $x_{i_3-1}>x_{i_3}$.
\end{itemize}
Indeed, $(i)$ is simply achieved by taking the largest $i_1$ and
the smallest $i_3$ such that $x_{i_1}x_{i_2}x_{i_3}\simeq 212$.
Furthermore, since $x_{i_1}=x_{i_3}$, we have
$x_{i_3}\notin\nub(x)=\asctops(x)$.
Hence $x_{i_3-1}\ge x_{i_3}$ and, due to $(i)$, $x_{i_3-1}>x_{i_3}$.
Finally, the Burge transpose maps the entries
$x_{i_1}x_{i_2}x_{i_3-1}x_{i_3}$ to an occurrence of $\alpha$ in $p$.
Indeed $x_{i_1}x_{i_2}x_{i_3-1}x_{i_3}\simeq 2132$ is mapped by $T$
to an occurrence of $2413$. The entries playing the role of $4$
and $1$ are in consecutive positions due to $(i)$;
and the $4$ and the $3$ are consecutive in value due to $(ii)$.

Next suppose that $p$ contains $\alpha$. We show that $x$
contains $212$.
The classical pattern underlying $\alpha$ is $2413$.
Due to Theorem~\ref{transport_theorem_modasc_fish},
since the Fishburn basis of $2413$ is
$$
B_{2413}=\lbrace x\in\Cay:\;\gamma(x)=2413\rbrace
=\lbrace 2132,3142\rbrace,
$$
we have
$$
\F(2413)=\gamma\bigl(\Modasc(2132,3142)\bigl).
$$
Hence, since $p$ contains $2413$, $x$ contains $2132$ or $3142$.
If $x$ contains $2132$, then we are done since $\Modasc(212)=\Modasc(2132)$
by Proposition~\ref{212_eq_2132}. On the other hand, if $x$
contains $3142$, then an occurrence of $212$ can be obtained by taking
the leftmost copy of the entry that plays the role of $2$, which must
precede the $1$ due to the shaded regions defining $\alpha$.
\end{proof}

\begin{corollary}
The sets of Fishburn permutations
$$
\F(\alpha),\quad
\F(\beta_1,\beta_2),\quad
$$
are counted by the Bell numbers. Furthermore, the distribution of
the number of occurrences of the pattern
$$
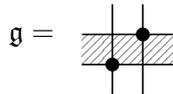
\begin{matrix}
\mathfrak{g}=
&
\begin{tikzpicture}[scale=0.40, baseline=20.5pt]
\fill[NE-lines] (0,1) rectangle (3,2);
\draw [semithick] (0.001,0.001) grid (2.999,2.999);
\filldraw (1,1) circle (6pt);
\filldraw (2,2) circle (6pt);
\end{tikzpicture}
\end{matrix}
$$
is the reverse of the distribution
of blocks on set partitions.
\end{corollary}
\begin{proof}
The first part of the statement follows immediately by
Theorem~\ref{proof_212}, Theorem~\ref{proof_2213_2231},
Proposition~\ref{fish_2213} and Proposition~\ref{fish_212}.
The second part follows from the same results since $\gamma$ maps
each strict ascent in a modified ascent sequence $x$ to an occurrence
of $\mathfrak{g}$ in the corresponding Fishburn permutation $\gamma(x)$.
\end{proof}

The current author~\cite{CeModasc} has recently solved
Conjecture~\ref{conj_ds} for the remaining pattern $2321$.
By Proposition~\ref{fish_2321}, the set $\F(\delta_1,\delta_2)$ is also
counted by the Bell numbers.

\section{Final remarks}\label{sec_final_remarks}

In this paper, we proved that modified ascent sequences
avoiding any of the patterns in $\{212,1212,2132,12132,2213,2231\}$
are counted by the Bell numbers; we also showed that the
distribution of strict ascents (or, equivalently, of the maximum
value) is the reverse of the number of blocks on set partitions.
The pattern $2321$ was solved recently by the current
author~\cite{CeModasc}, thus fully answering Conjecture~\ref{conj_ds}.

Claesson and the current author~\cite{CC2} have recently introduced
Fishburn trees to clarify the bijections relating modified
ascent sequences, Fishburn matrices and unlabeled $(\twoplustwo)$-free posets.
Under these maps, the sets $\Modasc(212)$, $\Modasc(2213)$,
$\Modasc(2231)$ and $\Modasc(2321)$ determine sets of matrices, trees and posets that
are counted by the Bell numbers: can we find independent and interesting
description of these sets, e.g. in terms of properties defined directly
on each of these structures?

Pattern avoiding modified ascent sequences are related to several
other combinatorial structures. One notable instance is given by
the set $\Modasc(2312,3412)$. The set of Fishburn permutations
corresponding to $\Modasc(2312,3412)$ is $\F(3412)$~\cite{CC}.
Furthermore, the pair of statistics right-to-left maxima and
right-to-left minima on $\F(3412)$ seems to have the same
distribution as the pair left-to-right maxima and right-to-left maxima
over the set of $312$-sortable permutations~\cite{CCF}.
The enumeration of all these sets (see also A202062~\cite{Sl}) is
still unknown.

Sequences in context:\\
A000670, A000110, A022493, A137251, A005493, A259691, A202062.

\end{document}